\documentclass[11pt,leqno]{article}

\usepackage{fullpage}
\usepackage[T1]{fontenc}
\usepackage{amsmath,amsfonts,amssymb,amsthm}
\usepackage{pb-diagram,graphicx}
\usepackage{epsfig}
\usepackage{graphics}
\usepackage[pst-3d]{pstricks}
\usepackage{psfrag}
\usepackage{latexsym}
\usepackage{multicol}
\usepackage{times}
\usepackage{bbold}
\usepackage{dsfont}

\newtheorem{definition}{Definition}

\newtheorem{proposition}{Proposition}
\newtheorem{theorem}{Theorem}
\newtheorem{remark}{Remark}

\newcommand{\e }{ \varepsilon}
\newcommand{\om }{ \omega}
\newcommand{\Om }{ \Omega}
\newcommand{\Oms }{ \Omega^\star}

\newcommand{\dn }{ \partial_n}

\newcommand{\OO}{\mathcal{O}}

\newcommand{\B}{\mathcal{B}}
\newcommand{\D }{\mathcal{D}}
\newcommand{\be }{\beta_\e}
\newcommand{\jj }{J}

\newcommand{\wsa}{\widetilde\Sigma_2(a)}
\newcommand{\wsb}{\widetilde\Sigma_4(a)}
\newcommand{\wsc}{\widetilde\Sigma_3(a)}

\newcommand{\ub }{u_{2,\e}}

\newcommand{\p }{\xi}

\title{On a Bernoulli problem with geometric constraints}
\author{Antoine Laurain\\ Karl-Franzens-University of Graz,\\ Department of Mathematics and Scientific Computing, \\
Heinrichstrasse 36, A-8010 Graz, Austria \\
E-mail address: Antoine.Laurain@uni-graz.at
 \and Yannick Privat\\ 
IRMAR, ENS Cachan Bretagne, Univ. Rennes 1, CNRS, UEB,\\
  av. Robert Schuman, F-35170 Bruz, France, \\
E-mail address: Yannick.Privat@bretagne.ens-cachan.fr}

\begin{document}

\maketitle

\vspace*{0.4cm} {\bf Abstract.} 
A Bernoulli free boundary problem with geometrical constraints is studied. The domain $\Om$ is constrained to lie in the half space determined by $x_1\geq 0$ and its boundary  to contain a segment of the hyperplane $\{x_1=0\}$ where non-homogeneous Dirichlet conditions are imposed. We are then looking for the solution of a partial differential equation satisfying a Dirichlet and a Neumann boundary condition simultaneously on the free boundary. The existence and uniqueness of a solution have already been addressed and this paper is devoted first to the study of geometric and asymptotic properties of the solution and then to the numerical treatment of the problem using a shape optimization formulation. The major difficulty and originality of this paper lies in the treatment of the geometric constraints.
\vspace*{0.3cm}



\section{Introduction}
Let $(0,x_1,...,x_N)$ be a system of Cartesian coordinates in $\mathds{R}^N$ with $N\geq 2$. We set  $\mathds{R}^N_+=\{\mathds{R}^N: x_1>0\}$. Let $K$ be a smooth, bounded and convex set such that $K$ is included in the hyperplane $\{x_1=0\}$. We define a set of {\it admissible shapes} $\OO$ as
$$\OO=\{\Om\mbox{ open and convex}, K\subset\partial\Om\}. $$
We are looking for a domain $\Om\in\OO$, and for a function $u:\Om\to\mathds{R}$ such that the following over-determined system
\begin{align}
 \label{1.1}-\Delta u & =  0 \quad  \mbox{in} \  \Om , \\
 \label{1.2}u & =  1 \quad  \mbox{on}\  K ,\\
 \label{1.3}u & =  0 \quad  \mbox{on}\  \partial\Om\setminus K ,\\
 \label{1.4}|\nabla u| & =1 \quad  \mbox{on} \  \Gamma:=(\partial\Om\setminus K)\cap\mathds{R}^N_+ 
\end{align}
has a solution; see Figure \ref{omfig} for a sketch of the geometry. Problem \eqref{1.1}-\eqref{1.4} is a {\it free boundary problem} in the sense that it admits a solution only for particular geometries of the domain $\Om$. The set $\Gamma$ is the so-called {\it free boundary} we are looking for. Therefore, the problem is formulated as
\begin{equation}
(\mathcal{F}):\mbox{ Find }\Om\in\OO\mbox{ such that problem }\eqref{1.1}-\eqref{1.4}\mbox{ has a solution.}
\end{equation}
This problem arises from various areas, for instance shape optimization, fluid dynamics, electrochemistry and electromagnetics, as explained in \cite{MR605427,MR1260974,MR761737,MR1776102}. For applications in $N$ diffusion, we refer to \cite{MR0140343} and for the deformation plasticity see \cite{MR760209}.

\begin{figure}
\begin{center}
\psfrag{e}{$\Delta u=0$}
\psfrag{o}{$\Om$}
\psfrag{d1}{$u=1$}
\psfrag{d0}{$u=0$}
\psfrag{n}{$|\nabla u|  =1 $}
\psfrag{l}{$L$}
\psfrag{k}{$K$}
\psfrag{g}{$\Gamma$}
\psfrag{x}{$x_1$}
\psfrag{y}{$x_2$}
\includegraphics[width=6cm]{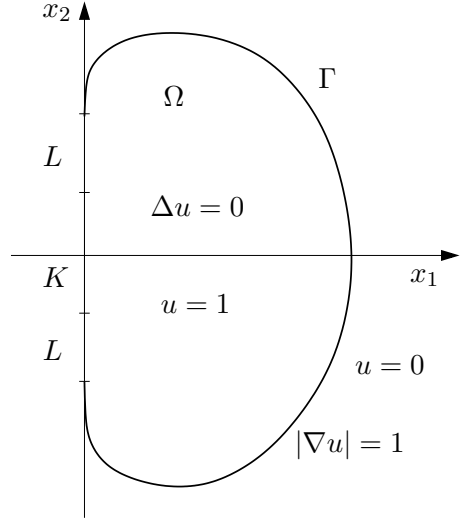}
\end{center}
\caption{The domain $\Om$ in dimension two.}\label{omfig}
\end{figure}

\par For our purposes it is convenient to introduce the set $L:=(\partial\Om\setminus K )\cap\{x_1=0\}$. Problems of the type $(\mathcal{F})$ may or may not, in general, have solutions, but it was already proved in \cite{MR2338115} that there exists a unique solution to $(\mathcal{F})$ in the class $\OO$. Further we will denote $\Oms$ this solution. In addition,  it is shown in \cite{MR2338115} that $\partial\Oms$ is $C^{2+\alpha}$ for any $0<\alpha<1$, that the free boundary $\partial\Oms\setminus K$ meets the fixed boundary $K$ tangentially and that $L^\star=(\partial\Oms\setminus K )\cap\{x_1=0\}$ is not empty.\\

In the literature, much attention has been devoted to the Bernoulli problem in the geometric configuration where the boundary $\partial\Om$ is composed of two connected components and such that $\Om$ is connected but not simply connected, (for instance for a ring-shaped $\Om$), or for a finite union of such domains; we refer to \cite{Beurling,FluRu97} for a review of theoretical results and to \cite{MR2149216,MR2511644,MR2013364,MR2183542,MR2306261} for a description of several numerical methods for these problems. In this configuration, one distinguishes the interior Bernoulli problem where the additional boundary condition similar to \eqref{1.4} is on the inner boundary, from the exterior Bernoulli problem where the additional boundary condition is on the outer boundary. The problem studied in this paper can be seen as a ``limit'' problem of the exterior boundary problem described in \cite{MR1752296}, since $\partial\Om$ has one connected component and $\Om$ is simply connected.\\

\par In comparison to the standard Bernoulli problems,  $(\mathcal{F})$ presents several additional distinctive features, both from the theoretical and numerical point of view. The difficulties here stem from the particular geometric setting. Indeed, the constraint $\Om\subset\mathds{R}^N_+$ is such that the hyperplane $\{x_1=0\}$ behaves like an obstacle for the domain $\Om$ and the free boundary $\partial\Om\setminus K$. It is clear from the results in \cite{MR2338115} that this constraint will be active as the optimal set $L^\star=(\partial\Oms\setminus K )\cap\{x_1=0\}$ is not empty. This type of constraint is difficult to deal with in shape optimization and there has been very few attempts, if any, at solving these problems.\\

\par From the theoretical point of view, the difficulties are apparent in \cite{MR2338115}, but a proof technique used for the standard Bernoulli problem may be adapted to our particular setup. 
Indeed, a Beurling's technique and a Perron argument were used, in the same way as in \cite{MR1752296,MR1777029,MR1885658}. 
\par Nevertheless, the proof of the existence and uniqueness of the free boundary is mainly theoretical and no numerical algorithm may be deduced to construct $\Gamma$.  From the numerical point of view, several problems arise that will be discussed in the next sections. The main issue is that $\Gamma$ is a free boundary but the set $L=(\partial\Om\setminus K )\cap\{x_1=0\}$ is a "free" set as well, in the sense that its length is unknown and should be obtained through the optimization process. In other words, the interface between $L$ and $\Gamma$ has to be determined and this creates a major difficulty for the numerical resolution.\\

\par The aim of this paper is twofold: on one hand we perform a detailed analysis of the geometrical properties of the free boundary  $\Gamma$ and in particular we are interested in the dependence of $\Gamma$ on $K$. On the other hand, we introduce an efficient algorithm in order to compute a numerical approximation of $\Om$. In this way we perform a complete analysis of the problem.\\
\indent First of all, using standard techniques for free boundary problems, we prove  symmetry and monotonicity properties of the free boundary.  These results are used further to prove the main theoretical result of the section in Subsection \ref{SectAsympt}, where the asymptotic behavior of the free boundary, as the length of the subset $K$ of the boundary diverges, is exhibited. The proof is based on a judicious cut-out of the optimal domain and on estimates of the solution of the associated partial differential equation to derive the variational formulation driving the solution of the ``limit problem''. 
\indent Secondly, we give a numerical algorithm for a numerical approximation of $\Om$. To determine the free boundary we use a shape optimization approach as in \cite{MR2511644,MR2013364,MR2183542}, where a penalization of one of the boundary condition in \eqref{1.1}-\eqref{1.4} using a shape functional is introduced. However, the original contribution of this paper regarding the numerical algorithm  comes from the way how  the "free" part $L$ of the boundary is handled. Indeed, it has been proved in the theoretical study presented in \cite{MR2338115} that the set $L=(\partial \Om\backslash K)\cap \{x_1=0\}$ has nonzero length. The only equation satisfied on $L$ is the Dirichlet condition, and a singularity naturally appears in the solution at the interface between $K$ and $L$ during the optimization process, due to the jump in boundary conditions. This singularity is a major issue for numerical algorithms: the usual numerical approaches for standard Bernoulli free boundary problems  \cite{MR2511644,MR2013364,MR2183542,MR1678201} cannot be used and a specific methodology has to be developed. A solution proposed in this paper consists in introducing a partial differential equation with special Robin boundary conditions depending on an asymptotically small parameter $\e$ and approximating the solution of the free boundary problem. We then prove in Theorem \ref{thm1} the convergence  of the approximate solution to the solution of the free boundary problem, as $\e$ goes to zero. In doing so we show the efficiency of a numerical algorithm that may be easily adapted to solve other problems where the free boundary meets a fixed boundary as well as free boundary problems with geometrical constraints or jumps in boundary conditions. Our implementation is based on a standard parameterization of the boundary using  splines. Numerical results show the efficiency of the approach.

\par The paper is organized as follows. Section \ref{SectionShape} is devoted to recalling basic concepts of shape sensitivity analysis. In Section \ref{geomSect}, we provide qualitative properties of the free boundary $\Gamma$, precisely we exhibit  symmetry and a monotonicity property  with respect to the length of the set $K$ as well as asymptotic properties of $\Gamma$. In Section \ref{SectPena},  the shape optimization approach for the resolution of the free boundary problem and a penalization of the p.d.e. to handle the jump in boundary conditions are introduced. In section \ref{SectDer}, the shape derivative of the functionals are computed and used in the numerical simulations of  $(\mathcal{F})$ in Section \ref{numScheme} and \ref{numres}.
\section{Shape sensitivity analysis}\label{SectionShape}

To solve the free boundary problem $(\mathcal{F})$, we formulate it as a shape optimization problem, i.e. as the minimization of a functional which depends on the geometry of the domains $\Om\subset\OO$. In this way we may study the sensitivity with respect to perturbations of the shape and use it in a numerical algorithm. The shape sensitivity analysis is also useful to study the dependence of $\Oms$ on the length of $K$, and in particular to derive the monotonicity of the domain  $\Oms$ with respect to  the length of $K$.\\

The major difficulty in dealing with sets of shapes is that they do not have a vector space structure. In order to be able to define shape derivatives and study the sensitivity of shape functionals, we need to construct such a structure for the shape spaces. In the literature, this is done by considering perturbations of an initial domain; see \cite{MR1855817,MR2512810,MR1215733}.\\

Therefore, essentially two types of domain perturbations are considered in general. The first one is a {\it method of perturbation of the identity operator}, the second one, the {\it velocity} or {\it speed method} is based on the deformation obtained by the flow of a velocity field. The speed method is more general than the method of perturbation of the identity operator, and the equivalence between deformations obtained by a family of transformations and deformations obtained by the flow of velocity field may be shown \cite{MR1855817,MR1215733}. The method of perturbation of the identity operator is a particular kind of domain transformation, and in this paper the main results will be given using a simplified speed method, but we point out that using one or the other is rather a matter of preference as several classical textbooks and authors rely on the method of perturbation of the identity operator as well. \\

For the presentation of the speed method, we mainly rely on the presentations in \cite{MR1855817,MR1215733}.  We also restrict ourselves to shape perturbations by {\it autonomous} vector fields, i.e. time-independent vector fields. Let $V:\mathds{R}^N\to\mathds{R}^N$ be an autonomous vector field. Assume that 
\begin{equation}
\label{vlip}V\in\D^k(\mathds{R}^N,\mathds{R}^N)=\{V\in\mathcal{C}^k(\mathds{R}^N,\mathds{R}^N),\ V\mbox{ has compact support} \},
\end{equation}
with $k\geq 0$.
\par For $\tau>0$, we introduce a family of transformations $T_t(V)(X)=x(t,X)$ as the solution to the ordinary differential equation
\begin{equation}\label{2.1.1}
\left\{\begin{array}{rclr}
\displaystyle \frac{d}{dt}x(t,X) & =& V(x(t,X)), & 0<t<\tau ,\\
x(0,X) &=& X\in\mathds{R}^N . &
\end{array}\right .
\end{equation}
For $\tau$ sufficiently small, the system \eqref{2.1.1} has a unique solution \cite{MR1215733}. The mapping $T_t$ allows to define a family of domains $\Om_t=T_t(V)(\Om)$ which may be used for the differentiation of the shape functional. We refer to  \cite[Chapter 7]{MR1855817}  and \cite[Theorem 2.16]{MR1215733} for Theorems establishing the regularity of transformations $T_t$.\\

It is assumed  that the shape functional $J(\Om)$ is well-defined for any measurable set $\Om\subset\mathds{R}^N$. We introduce the following notions of differentiability with respect to the shape
\begin{definition}[Eulerian semiderivative]\label{2.1.7}
Let $V\in\D^k(\mathds{R}^N,\mathds{R}^N)$ with $k\geq 0$, the {\it Eulerian semiderivative} of the shape functional $J(\Om)$ at $\Om$ in the direction $V$ is defined as the limit
\begin{equation}\label{2.1.8}
dJ(\Om;V)=\lim_{t\searrow 0}\frac{J(\Om_t)-J(\Om)}{t},
\end{equation}
when the limit exists and is finite.
\end{definition}

\begin{definition}[Shape Differentiability]\label{2.1.9}
The functional $J(\Om)$ is {\it shape differentiable} (or {\it differentiable} for simplicity) at $\Om$ if it has a Eulerian semiderivative at $\Om$ in all directions $V$ and the map
\begin{equation}\label{2.1.10}
V\mapsto dJ(\Om,V)
\end{equation}
is linear and continuous from $\D^k(\mathds{R}^N,\mathds{R}^N)$ into $\mathds{R}$. The map \eqref{2.1.10} is then sometimes denoted $\nabla J(\Om)$ and referred to as the {\it shape gradient} of $J$ and we have
\begin{equation}\label{2.1.10b}
dJ(\Om,V)=\langle\nabla J(\Om),V\rangle_{\D^{-k}(\mathds{R}^N,\mathds{R}^N),\D^k(\mathds{R}^N,\mathds{R}^N)}
\end{equation}
\end{definition}
When the data is smooth enough, i.e. when the boundary of the domain $\Om$ and the velocity field $V$ are smooth enough (this will be specified later on), the shape derivative has a particular structure: it is concentrated on the boundary $\partial\Om$ and depends only on the normal component of the velocity field $V$ on the boundary $\partial\Om$.   This result, often called {\it structure theorem} or {\it Hadamard Formula}, is fundamental in shape optimization and will be observed in Theorem \ref{thm_shapeder}.

\section{Geometric properties and asymptotic behaviour}\label{geomSect}

In shape optimization, once the existence and maybe uniqueness of an optimal domain have been obtained, an explicit representation of the domain, using a parameterization for instance usually cannot be achieved, except in some particular cases, for instance if the optimal domain has a simple shape such as a ball, ellipse or a regular polygon. On the other hand, it is usually possible to determine important geometric properties of the optimum, such as symmetry, connectivity, convexity for instance. In this section we show first of all that the optimal domain is symmetric with respect to the perpendicular bissector of the segment $K$, using a symmetrization argument. Then, we are interested in the asymptotic behaviour of the solution as the length of $K$ goes to infinity. We are able to show that the optimal domain $\Oms$ is monotonically increasing for the inclusion when  the length of $K$ increases, and that $\Oms$ converges, in a sense that will be given in Theorem \ref{strip}, to the infinite strip $(0,1)\times\mathds{R}$.
\par The proofs presented in Subsections \ref{symm} and \ref{subsectmono} are quite standard and similar ideas of proofs may be found e.g. in \cite{MR2512810,MR1752296,MR1777029,MR1885658}.
\subsection{Symmetry}\label{symm}
In this subsection, we derive a symmetry property of the free boundary. The interest of such a remark is intrinsic and appears useful from a numerical point of view too, for instance to test the efficiency of the chosen algorithm.
\par  In the two-dimensional case, we have the following result of symmetry:
 \begin{proposition}
 Let $\Omega ^\star$ be the solution of the free boundary problem
 \eqref{1.1}-\eqref{1.4}. Assume, without loss of generality that $(Ox_1)$ is the
 perpendicular bissector of $K$. Then, $\Oms$ is symmetric with
 respect to $(Ox_1)$.
 \end{proposition}
\begin{proof}
Like often, this proof is based on a symmetrization argument. It may be noticed that, according to the result stated in \cite[Theorem 1]{MR2338115}, $\Om$ is the unique solution of the overdetermined optimization problem
\begin{displaymath}\label{minSymm}
(\B_0):\left\{\begin{array}{ll}
\mbox{\text{minimize}} &\jj(\Om,u)\\ 
\mbox{\text{subject to}} & \Om\in\OO,  u\in H(\Om),
\end{array}\right.
\end{displaymath}
where 
$$
H(\Om)=\{u\in H^1(\Om ), u=1\textrm{ on }K, u=0\textrm{ on }\partial\Om \backslash K\textrm{ and }|\nabla u|=1\textrm{ on }\Gamma\},
$$
and
$$
\jj(\Om,u)=\displaystyle \int_{\Om}|\nabla u(x)|^2dx. 
$$
From now on, $\Om^\star$ will denote the unique solution of $(\B_0)$, $K$ being fixed.
We denote by $\widehat \Om$ the Steiner symmetrization of $\Om$ with respect to the hyperplane $x_2=0$, i.e.
$$
\widehat \Om =\left\{x=(x',x_2)\textrm{ such that }-\frac{1}{2}|\Om (x')|<x_2<\frac{1}{2}|\Om (x')|,x'\in \Om'\right\},
$$
where 
$$
\Om '=\{x'\in \mathds{R}\textrm{ such that there exists }x_2\textrm{ with }(x',x_2)\in \Om^\star\}
$$
and
$$
\Om (x')=\{x_2\in \mathds{R}\textrm{ such that }(x',x_2)\in \Om\}, \ x'\in \Om'.
$$
By construction, $\widehat\Om$ is symmetric with respect to the $(Ox_1)$ axis. Let us also introduce $\widehat u$, defined by
$$
\widehat u:x\in \widehat \Om \mapsto \sup \{c\textrm{ such that }x\in \widehat{\om^\star(c)}\},
$$
where $\om^\star (c)=\{x\in \Om^\star:u(x)\geq c\}$.
Then, one may verify that $\widehat u\in H(\widehat \Om)$ and  Poly\`a's inequality (see \cite{MR2512810}) yields
$$
J(\widehat\Om, \widehat u)\leq J(\Om^\star,u^\star).
$$
Since $(\Om^\star,u^\star)$ is a minimizer of $J$ and using the uniqueness of the solution of $(\B_0)$, we get $\Om^\star = \widehat \Om$.
\end{proof}
\begin{remark}
This proof yields in addition that the direction of the normal vector at the intersection of $\Gamma$ and $(Ox_1)$ is $(Ox_1)$. 
 \end{remark}

\subsection{Monotonicity}\label{subsectmono}

In this subsection, we show that $\Om^\star$ is monotonically increasing for the inclusion when the length of $K$ increases. For a given $a>0$, define $K_a=\{0\}\times[-a,a]$. Let $(\mathcal{F}_a)$ denote problem $(\mathcal{F})$ with $K_a$ instead of $K$ and denote $\Om_a$ and $u_a$ the corresponding solutions. We have the following result on the monotonicity of $\Om_a$ with respect to $a$.
\begin{theorem}\label{OmMono}
Let $0<a<b$, then $\Om_{a}\subset\Om_{b}$. 
\end{theorem}
\begin{proof}
According to \cite{MR2338115}, $(\mathcal{F}_a)$ has a solution for every $a>0$ and $\partial\Om_a$ is $C^{2+\alpha}$, $0<\alpha<1$. 
We argue by contradiction, assuming that $\Om_a \not\subset \Om_b$. Introduce, for $t\geq 1$, the set
$$
\Om_t =\{x\in \Om_a : tx\in \Om_a\}.
$$
We also denote by $K_t:=\{0\}\times [-ta,ta]$ and $\Gamma_t :=\partial \Om_t\backslash (\partial \Om_t\cap (Ox_2))$.  The domain
$\Om_t$ is obviously a convex set included in $\Om_a$ for $t\geq 1$. Now denote 
$$
t_{min}:= \inf \{t\geq 1, \Om_t\subset\Om_b\}.
$$
On one hand, $\Om_a\subset\Om_b$ is equivalent to $t_{min}=1$. On the other hand, if $\Om_a\not\subset\Om_b$, then $t_{min}>1$ and for $t$ large enough, we clearly have $\Om_t\subset\Om_b$, therefore $t_{min}$ is finite. In addition, if $\Om_a\not\subset\Om_b$ we have $\Gamma_{t_{min}}\cap\Gamma_b\neq\emptyset$. Now, choose $y\in \Gamma_{t_{min}}\cap\Gamma_b$. Let us introduce 
$$
u_{t_{min}}:x\in \Om_t\mapsto u_a(t_{min}x).
$$
Then, $u_{t_{min}}$ verifies
\begin{align*}
-\Delta u_{t_{min}} & =0\quad  \textrm{in }\Om_{t_{min}},\\
u_{t_{min}} & =1\quad \textrm{on }K_{t_{min}},\\
u_{t_{min}} & =0\quad  \textrm{on }\Gamma_{t_{min}},
\end{align*}
so that, in view of $\Om_{t_{min}}\subset\Om_b$ and $K_{t_{min}}\subset K_b$, the maximum principle yields $u_b\geq u_{t_{min}}$ in $\Om_{t_{min}}$. Consequently, the function $h=u_b-u_{t_{min}}$ is harmonic in $\Om_{t_{min}}$, and since $h (y)=0$, $h$ reaches its lower bound at $y$. Applying Hopf's lemma (see \cite{evans}) thus yields $\partial _n h (y)<0$ so that $|\nabla u_b(y)|\geq |\nabla u_{t_{min}}(y)|$. Hence,
$$
1=|\nabla u_b(y)|\geq |\nabla u_{t_{min}}(y)|=t_{min}>1,
$$
which is absurd. Therefore we necessarily have $t_{min}=1$ and $\Om_a\subset\Om_b$.
\end{proof}

\subsection{Asymptotic behaviour}\label{SectAsympt}

We may now use the symmetry property of the free boundary to obtain the  asymptotic properties of $\Om_a$ when the length of $K$ goes to infinity, i.e. we are interested in the behaviour of the free boundary $\Gamma_a$ as $a\to\infty$. 
\par Let us say one word on our motivations for studying such a problem. First, this problem can be seen as a limit problem of the ``unbounded case'' studied in \cite[Section 5]{MR1885658} relative to the one phase free boundary problem for the $p$-Laplacian with non-constant Bernoulli boundary condition. Second, let us notice that the change of variable $x'=x/a$ and $y'=y/a$ transforms the free boundary \eqref{1.1}-\eqref{1.4} problem into
\begin{eqnarray}
 -\Delta z & =&  0 \quad  \mbox{in} \  \Om ,\label{homot1} \\
 z & =&  1 \quad  \mbox{on}\  K_1 ,\\
 z & = & 0 \quad  \mbox{on}\  \partial\Om\setminus K_1 ,\\
 |\nabla z| & =& a \quad  \mbox{on} \  \Gamma=(\partial\Om\setminus K_1)\cap\mathds{R}^N_+ ,\label{homot4}
\end{eqnarray}
which proves that the solution of \eqref{homot1}-\eqref{homot4} is $h_{1/a}(\Omega_a)$, where $h_{1/a}$ denotes the homothety centered at the origin, with ratio $1/a$. Hence such a study permits also to study the role of the Lagrange multiplier associated with the volume constraint of the problem
$$
\left\{\begin{array}{l}
\min C(\Omega)\textrm{ where }C(\Omega)=\min \left\{\frac{1}{2}\int_\Omega |\nabla u_\Omega|^2,u=1\textrm{ on }K_1, \ u=0\textrm{ on }\partial \Omega\backslash K_1\right\} \\
\Omega\textrm{ quasi-open}, \ |\Omega|=m,
\end{array}
\right.
$$
since, as enlightened in \cite[Chapter 6]{MR2512810}, the optimal domain is the solution of \eqref{homot1}-\eqref{homot4} for a certain constant $a>0$. The study presented in this section permits to link the Lagrange multiplier to the constant $m$ appearing in the volume constraint and to get some information on the limit case $a\to +\infty$.
\ \\

\par We actually show that $\Gamma_a$ converges, in an appropriate sense, to the line parallel to $K_a$ and passing through the point $(1,0)$. Let us introduce the infinite open strip
$$
S=]0,1[ \times \mathds{R},
$$
and the open, bounded rectangle
$$
R(b)=]0,1[ \times ]-b,b[ \subset S. 
$$
Let 
$$
u_S:x\in S\mapsto 1-x_1.
$$ 
Observe that, since $\Om_a$ is solution of the free boundary problem \eqref{1.1}-\eqref{1.4}, the curve $\Gamma_a\cap \{-a\leq x_2\leq a\}$ is the graph of a concave $\mathcal{C}^{2,\alpha}$ function $x_2\mapsto \psi_a (x_2)$ on $[-a,a]$.
We have the following result
\begin{theorem}\label{strip} The domain $\Om_a$ converges to the strip $S$ in the sense that for all $b>0$, we have
\begin{equation}
\label{ab1}\psi_a \to 1, \textrm{ uniformly in }[-b,b]\textrm{, as }a\to +\infty.
\end{equation}
We also have the convergence 
$$
u_a\to u_S \quad\mbox{ in }H^1(R(b))\quad\mbox{ as } a\to\infty,
$$
for the solution $u_a$ of \eqref{1.1}-\eqref{1.4}.
\end{theorem}
\begin{proof}
Let us introduce the function
$$
v_a(x_1)=u_a(x_1,0) .
$$
According to \cite[Proposition 5.4.12]{MR2512810}, we have for a domain $\Om$ of class $\mathcal{C}^2$ and $u:\overline{\Om}\to\mathds{R}$ of class $\mathcal{C}^2$
\begin{equation}
\label{ab2}\Delta u=\Delta_\Gamma u +\mathcal{H}\dn u+\dn^2 u ,
\end{equation}
where $\Delta_\Gamma u$ denotes the Laplace-Beltrami operator. Applying formula \eqref{ab2} in the domains
$$
\om_a(c) :=\{x\in\Om_a,u_a(x)> c\}, 
$$
we get $\Delta_\Gamma u_a=0$ on $\partial\om_a(c)$,  $\Delta u_a=0$ due to \eqref{1.1} and thus
\begin{equation}
\label{ab2a} \dn^2 u_a = -\mathcal{H}_a\dn u_a \quad\mbox{ on } \partial\om_a(c),
\end{equation}
where $n$ is the outer unit normal vector to $\om_a(c)$ and $\mathcal{H}_a(x)$ denotes here the curvature of $\partial\om_a(c)$ at a point $x\in\partial\om_a(c)$. Thanks to the symmetry of $\Om_a$ with respect to the $x_1$-axis, we have $\dn u_a(x_1,0)= v_a'(x_1)$ and $\dn^2 u_a(x_1,0)=v_a''(x_1)$ for $x_1\geq 0$. According to \cite{MR2338115}, the sets $\om_a(c)$ are convex. Therefore $\mathcal{H}_a$ is positive on $\partial\om_a(c) $ and $v_a(x_1)$ is non-increasing. Thus
\begin{equation}
\label{ab2b} v_a''(x_1) = -\mathcal{H}_a(x_1,0)v_a'(x_1)\geq 0,
\end{equation}
which means that $v_a$ is convex. Let $m_a$ be such that $\Gamma_a\cap (\mathds{R}\times\{0\}) =(m_a,0)$, i.e. the first coordinate of the intersection of the $x_1$-axis and the free boundary $\Gamma_a$. The function $v_a$ satisfies
\begin{align}
\label{ab3.1}-v_a''(x_1) &\leq 0\quad \mbox{for } x_1\in]0,m_a[,\\
\label{ab3.2}v_a(0) & =1,\\
\label{ab3.3}v_a(m_a) & =0, \\
\label{ab3.4}v_a'(m_a) & =-1.
\end{align}
In view of \eqref{ab3.1}, $v_a$ is convex on $[0,m_a]$. Since $v_a(0)=1$ and $v_a(m_a)=0$, then 
$$
v_a(x_1)\leq 1-\frac{x_1}{m_a}.
$$
Furthermore, $m_a\leq 1$, otherwise, due to the convexity of $v_a$, the Neumann condition \eqref{ab3.4} would not be satisfied. Since $\Om_a$ is convex, this proves that $\Om_a\subset S$ and that $\Om_a$ is bounded.\\
Moreover, from Theorem \ref{OmMono}, the map $a\mapsto \Omega_a$ is nondecreasing with respect to the inclusion. It follows that the sequence $(m_a)$ is nondecreasing and bounded since $\Om _a\subset S$. Hence, $(m_a)$ converges to $m_\infty\leq 1$. \\
Let us define 
$$
u_\infty(x_1)=1-\frac{x_1}{m_\infty}.
$$
The previous remarks ensure that for every $a>0$, $v_a\leq u_\infty$.
\par Let $\mathcal{D}(a)$ be the line containing the points $(0,a)$ and $(\psi_a(b),b)$ and $\mathcal{T}(a)$ the line tangent to $\Gamma_a$ at $(\psi_a(b),b)$. Let $s_\mathcal{D}(a)$ and $s_\mathcal{T}(a)$ denote the slopes of $\mathcal{D}(a)$ and $\mathcal{T}(a)$, respectively. For a fixed $b\in (0,a)$, we have
$$
s_\mathcal{D}(a)=\frac{b-a}{\psi _a(b)}\to -\infty\quad \mbox{ as } a\to\infty,  
$$
since $0\leq \psi_a\leq 1$. Due to the convexity of $\Om_a$, we also have $s_\mathcal{T}(a)<s_\mathcal{D}(a)$. Therefore 
$$
s_\mathcal{T}(a)\to -\infty\quad \mbox{ as } a\to\infty  .
$$
Thus, the slopes of the tangents to $\Gamma_a$ go to infinity in $\Om_a\cap R(b)$. Furthermore, due to the concavity of the function $\psi_a$, we get, by construction of $\mathcal{D}(a)$,
$$
\frac{m_a}{a}(a-x_2)\leq \psi_a(x_2)\leq m_\infty, \ \forall a>0, \ \forall x_2\in [-b,b].
$$
Hence, we obtain the pointwise convergence result:
\begin{equation}\label{limPsi}
\lim_{a\to +\infty}\psi _a(x_2)=m_\infty, \ \forall x_2\in [-b,b],
\end{equation}
which proves the uniform convergence of $\psi_a$ to $m_\infty$ as $a\to +\infty$.\\
\par From now on, with a slight misuse of notation, $u_a$ will also denote its extension by zero to all of $S$. Finally, let us prove the convergence 
$$
u_a\to u_\infty \quad\mbox{ in }H^1(R_\infty(b)), \textrm{ as }a\to \infty,
$$
where $R_\infty(b)$ denotes the rectangle whose edges are: $\Sigma_1=\{0\}\times [-b,b]$, $\Sigma_2=[0,m_\infty]\times \{b\}$, $\Sigma_3=\{m_\infty\}\times [-b,b]$ and $\Sigma_4=[0,m_\infty]\times \{-b\}$. \\
According to the zero Dirichlet conditions on $\Sigma_3$ and using Poincar\'e's inequality, proving the $H^1$-convergence is equivalent to show that
\begin{equation}
\int_{R_\infty(b)} |\nabla (u_a-u_\infty)|^2 \to 0\quad\mbox{ as }a\to\infty .
\end{equation}
For our purposes, we introduce the curve $\wsa$ described by the points $X_{a,b}$ solutions of the following ordinary differential equation 
\begin{equation}\label{flow}
\left\{\begin{array}{ll}
\displaystyle \frac{d X_{a,b}}{dt}(t)=\nabla u _a(X_{a,b}(t)), & t>0, \\
X_{a,b}(0)=(0,b). &  
\end{array}
\right.
\end{equation}
 The curve $\wsa$ is naturally extended along its tangent outside of $\Om_a$. $\wsa$ can be seen as the curve originating at the point $(0,b)$ and perpendicular to the level set curves of $\Om_a$. We also introduce the curve $\widetilde\Sigma_4(a)$, symmetric to $\wsa$ with respect to the $x_1$-axis. $\widetilde\Sigma_4(a)$ is obviously the set of points $Y_{a,b}$ solutions of the following ordinary differential equation
\begin{equation}
\left\{\begin{array}{ll}
\displaystyle \frac{d Y_{a,b}}{dt}(t)=\nabla u _a(Y_{a,b}(t)), & t>0, \\
Y_{a,b}(0)=(0,-b). &  
\end{array}
\right.
\end{equation}
Then the set $Q_a(b)$ is defined as the region delimited by the $x_2$-axis on the left, the line parallel to the $x_2$-axis and passing through the point $(m_\infty ,0)$ on the right and the curves $\wsa$ and $\wsb$ at the top and bottom. We also introduce the set $\wsc := \overline{Q_a(b)}\cap (\{m_\infty\} \times\mathds{R})$.  See Figure \ref{rq} for a description of the sets $R_\infty(b)$ and $Q_a(b)$. \\

Since $R_\infty(b)\subset Q_a(b)$ (see Figure \ref{rq}), we have
$$
\int_{R_\infty(b)} |\nabla (u_a-u_\infty)|^2 \leq \int_{Q_a(b)} |\nabla (u_a-u_\infty)|^2 .
$$

\begin{figure}[!]
\begin{center}
\psfrag{s1}{$\Sigma_1$}
\psfrag{s2}{$\Sigma_2$}
\psfrag{s3}{$\Sigma_3$}
\psfrag{s4}{$\Sigma_4$}
\psfrag{s2t}{$\widetilde\Sigma_2(a)$}
\psfrag{s4t}{$\widetilde\Sigma_4(a)$}
\psfrag{r}{$R_\infty(b)$}
\psfrag{q}{$Q_a(b)$}
\psfrag{g}{$\Gamma_a$}
\psfrag{a}{$x_1$}
\psfrag{b}{$x_2$}
\includegraphics[width=8cm]{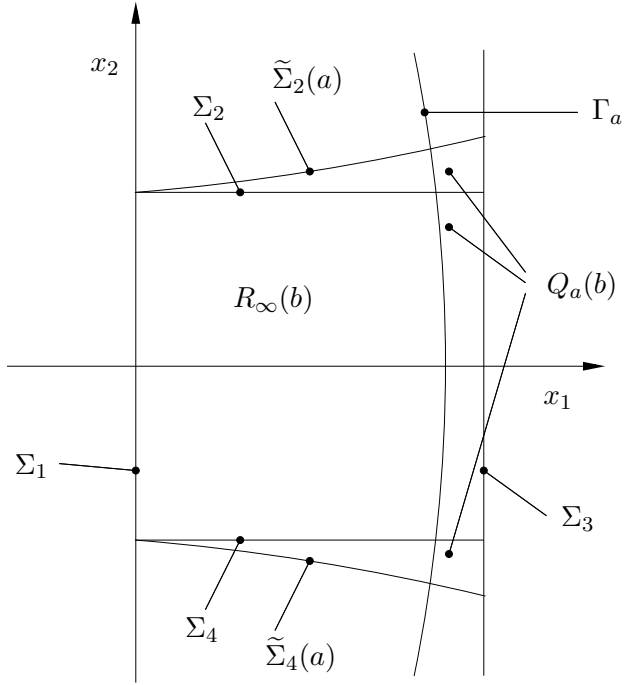}
\end{center}
\caption{The sets $R_\infty(b)$ and $Q_a(b)$.}\label{rq}
\end{figure}

Using Green's formula, we get
\begin{eqnarray*}
\int_{Q_a(b)} |\nabla (u_a-u_\infty)|^2 & = & \int_{Q_a(b)\cap\Om_a} |\nabla (u_a-u_\infty)|^2+ \int_{Q_a(b)\setminus \Om_a} |\nabla (u_a-u_\infty)|^2  \\
&=& -\int_{Q_a(b)\cap\Om_a} (u_a-u_\infty) \Delta (u_a-u_\infty)-\int_{Q_a(b)\setminus\Om_a} (u_a-u_\infty) \Delta (u_a-u_\infty)\\
& & +\int_{\partial Q_a(b)} (u_a-u_\infty) \partial_\nu (u_a-u_\infty)+\sum_\pm \int_{\Gamma_a\cap Q_a(b)} (u_a-u_\infty)\partial_{n^\pm} (u_a-u_\infty),
\end{eqnarray*}
where $\nu$ denotes the outer normal vector to $Q_a(b)$ on the boundary $\partial Q_a(b)$, $n$ is the outer normal vector to $\Om_a$ on the boundary $\Gamma_a$ and $\partial_{n^\pm}$ is the normal derivative on $\Gamma_a$ in the exterior or interior direction, the positive sign denoting the exterior direction to $\Om_a$. The functions $u_a$ and $u_\infty$ are harmonic, and using the various boundary conditions for $u_a$ and $u_\infty$ we get
\begin{align*}
\int_{Q_a(b)}|\nabla (u_a-u_\infty)|^2 =& \int_{\wsa\cup\wsb} (u_a-u_\infty) \partial_\nu (u_a-u_\infty) +\int_{\Gamma_a\cap Q_a(b)} u_\infty.
\end{align*}
According to \eqref{limPsi} and using $u_\infty=0$ on $\Sigma_3$, we get
$$ 
\int_{\Gamma_a\cap Q_a(b)} u_\infty = \int_{-b}^b u_\infty(\psi_a(x_2))\sqrt{1+\psi_a'(x_2)^2} dx_2 \to 0\quad\mbox{ as }a\to\infty ,
$$
where we have also used the fact that $\psi'_a(x_2)\to 0$ for all $x_2\in [-b,b]$.  The limit function $u_\infty$ depends only on $x_1$, thus we have $\partial_\nu u_\infty=0$ on  $\Sigma_2\cup\Sigma_4$. Denote now $\widetilde{\psi}_a:  [0, m_\infty]\to \mathds{R}$ the graph of $\wsa$ (which implies that $-\widetilde{\psi}_a$ is the graph of $\wsb$). The slope of the tangents to the level sets of $u_a$ converge to $-\infty$ as $a\to\infty$ in a similar way as for $\Gamma_a$, therefore $\partial_{x_2} u_a(x_1,\widetilde{\psi}_a(x_1))$ converges uniformly to $0$ in $[0,m_\infty]$ as $a\to\infty$, and in view of \eqref{flow} we have that $\widetilde{\psi}_a\to b$ uniformly in $ [0, m_\infty]$  and since $(u_a-u_\infty)$ is uniformly bounded in $\Om_a$ we have
\begin{equation}\label{gradQb}
\int_{\wsa\cup\wsb} (u_a-u_\infty)\partial_\nu u_\infty  \to 0\quad\mbox{ as }a\to\infty .
\end{equation}
In view of the definition of $Q_a(b)$, the outer normal vector $\nu$ to $Q_a(b)$ at a given point on $\wsa\cup\wsb$ is colinear with the tangent vector to the level set curve of $\Om_a$ passing though the same point. Therefore $\partial_\nu u_a=0$ on $\wsa\cup\wsb$ and we obtain finally
\begin{equation}\label{gradQb2}
0\leq
\int_{R_\infty(b)} |\nabla (u_a-u_\infty)|^2 \leq \int_{Q_a(b)} |\nabla (u_a-u_\infty)|^2\to 0\quad\mbox{ as }a\to\infty .
\end{equation}
\par The end of the proof consists in proving that $m_\infty=1$. Let us introduce the test function $\varphi$ as the solution of the partial differential equation
\begin{equation}
\left\{\begin{array}{ll}
\displaystyle -\Delta \varphi =0 & \textnormal{in } Q_a(b)  \\
\varphi =0 & \textnormal{on }\Sigma _1\cup\wsa\cup\wsb\\
\varphi =1 & \textnormal{on }\widetilde\Sigma _3(a).
\end{array}
\right.
\end{equation}
It can be noticed that $\varphi \in H^1(R_\infty(b))$.\\
Using Green's formula and the same notations as previously, we get
\begin{eqnarray*}
\int_{Q_a(b)} \nabla (u_a-u_\infty)\cdot \nabla \varphi & = & \int_{Q_a(b)\cap \Om_a} \nabla (u_a-u_\infty)\cdot \nabla \varphi +\int_{Q_a(b)\backslash \Om_a}\nabla (u_a-u_\infty)\cdot \nabla \varphi \\
 & = & -\int_{Q_a(b)\cap \Om_a} \varphi \Delta (u_a-u_\infty)-\int_{Q_a(b)\backslash \Om_a} \varphi \Delta (u_a-u_\infty)\\
& & +\int_{\partial Q_a(b)} \varphi \partial_\nu (u_a-u_\infty)+\sum_\pm \int_{\Gamma_a\cap Q_a(b)}  \varphi\partial_{n^\pm} (u_a-u_\infty)\\
& =& \int_{\wsa\cup\wsb} \varphi \partial_\nu (u_a-u_\infty) +\int_{\widetilde{\Sigma}_3(a)}\varphi \partial_n (u_a-u_\infty)-\int_{\Gamma_a\cap Q_a(b)} \varphi \\
& = & \int_{\widetilde{\Sigma}_3(a)}\frac{\varphi }{m_\infty}-\int_{\Gamma_a\cap Q_a(b)} \varphi.
\end{eqnarray*}
According to \eqref{limPsi}, and since we deduce from \eqref{gradQb2} that
$$
\int_{Q_a(b)} \nabla (u_a-u_\infty)\cdot \nabla \varphi \to 0\quad\mbox{ as }a\to\infty ,
$$
we get
$$
 \int_{\widetilde\Sigma _3(a)}\frac{\varphi }{m_\infty}-\int_{\widetilde\Sigma _3(a)} \varphi =0,
$$
which leads to
$$ 
\left(\frac{1}{m_\infty}-1\right)|\widetilde\Sigma _3(a)|=0.
$$
In other words, $m_\infty =1$, which ends the proof.
\end{proof}


\section{A penalization approach}\label{SectPena}

\subsection{Shape optimization problems}\label{sop}

From now on we will assume that $N=2$, i.e. we solve the problem in the plane. The problem for $N>2$ may be treated with the same technique, but the numerical implementation becomes tedious. A classical approach  to solve  the free boundary problem  is to penalize one of the boundary conditions in the over-determined system \eqref{1.1}-\eqref{1.4} within a shape optimization approach to find the free boundary.  For instance one may consider the well-posed problem 
\begin{align}
 \label{2.1}-\Delta u_1 & =  0 \quad  \mbox{in} \  \Om , \\
 \label{2.2}u_1 & =  1 \quad  \mbox{on}\  K ,\\
 \label{2.3}u_1 & =  0 \quad  \mbox{on}\  \partial\Om\setminus K .
\end{align}
and enforce the second boundary condition \eqref{1.4} by solving the problem
\begin{equation}\label{min1}
(\B_1):\left\{\begin{array}{ll}
\mbox{\text{minimize}} & J(\Om)\\ 
\mbox{\text{subject to}} & \Om\in\OO,\\
\end{array}\right.
\end{equation}
with the functional $J$ defined by
\begin{equation}
 \label{2.4} J(\Om)=\int_{\Gamma}(\dn u_1+1)^2\, d\Gamma.
\end{equation}
Indeed, using the maximum principle, one sees immediately that $u_1\geq 0$ in $\Om$ and since $u_1=0$ on $\partial\Om\setminus K$, we obtain $\dn u_1\leq 0$ on $\partial\Om\setminus K$. Thus $|\nabla u_1|=-\dn u_1$ on $\partial\Om\setminus K$ and the additional boundary condition \eqref{1.4} is equivalent to $\dn u_1=-1$ on  $\Gamma$. Hence, \eqref{2.4} corresponds to a penalization of condition \eqref{1.4}. On one hand, if we denote $u_1^\star$ the unique solution of \eqref{1.1}-\eqref{1.4} associated to the optimal set $\Oms$, we have 
$$J(\Oms)=0,$$ 
so that the minimization problem \eqref{min1} has a solution. On the other hand, if $J(\Oms)=0$, then  $|\nabla u_1^\star|\equiv 1$ on $\Gamma$ and therefore $u_1^\star$ is solution of \eqref{1.1}-\eqref{1.4}. Thus $(\mathcal{F})$ and $(\B_1)$ are equivalent.\\

Another possibility is to penalize boundary condition \eqref{1.3} instead of \eqref{1.4} as in $(\B_1)$, in which case we consider the problem
\begin{align}
 \label{6.1}-\Delta u_2 & =  0 \quad  \mbox{in} \  \Om , \\
 \label{6.2}u_2 & =  1 \quad  \mbox{on}\  K ,\\
 \label{6.3}u_2 & =  0 \quad  \mbox{on}\  L ,\\
 \label{6.4}\dn u_2 & =  -1 \quad  \mbox{on}\  \Gamma ,
\end{align}
and the shape optimization problem is
\begin{equation}\label{min2}
(\B_2):\left\{\begin{array}{ll}
\mbox{\text{minimize}} & J(\Om)\\ 
\mbox{\text{subject to}} & \Om\in\OO,\\
\end{array}\right.
\end{equation}
with the functional $J$ defined by
\begin{equation}
 \label{2.8} J(\Om)=\int_{\Gamma}(u_2)^2\, d\Gamma .
\end{equation}
Although the two approaches $(\B_1)$ and  $(\B_2)$ are completely satisfying from a theoretical point of view, it is numerically easier to minimize a domain integral rather than a boundary integral as in \eqref{2.4} and \eqref{2.8}. Therefore, a third classical approach is to solve
\begin{equation}\label{min3}
(\B_3):\left\{\begin{array}{ll}
\mbox{\text{minimize}} & J(\Om)\\ 
\mbox{\text{subject to}} & \Om\in\OO,\\
\end{array}\right.
\end{equation}
with the functional $J$ defined by
\begin{equation}
J(\Om)=\int_{\Om}(u_1-u_2)^2.
\end{equation}
For the standard Bernoulli problems \cite{Beurling,FluRu97}, solving $(\B_3)$ is an excellent approach as demonstrated in \cite{MR2511644,MR2013364,MR2183542}. However, we are still not quite satisfied with it in our case. Indeed, it is well-known that due to the jump in boundary conditions at the interface between $L$ and $\Gamma$ in \eqref{6.3}-\eqref{6.4}, the solution $u_2$ has a singular behaviour in the neighbourhood of this interface. To be more precise, let us define the points 
$$
\{A_1,A_2\}:=\overline{L}\cap\overline{\Gamma},
$$
and the polar coordinates $(r_i,\theta_i)$ with origin the points $A_i$, $i=1,2$, and such that $\theta_i=0$ corresponds to the semi-axis tangent to $\Gamma$; see Figure \ref{polarfig} for an illustration. Then, in the neighbourhood of $A_i$, $u_2$ has a singularity of the type 
$$
S_i(r_i,\theta_i)=c(A_i)\sqrt{r_i}\cos(\theta_i/2), 
$$
where $c(A_i)$ is the so-called {\it stress intensity factor} (see e.g. \cite{MR775683,MR0226187}). \\

\begin{figure}
\begin{center}
\psfrag{a}{$\Gamma$}
\psfrag{b}{$L$}
\psfrag{c}{$A_i$}
\psfrag{d}{$\theta_i=0$}
\psfrag{e}{$\theta_i=\pi$}
\includegraphics[width=6cm]{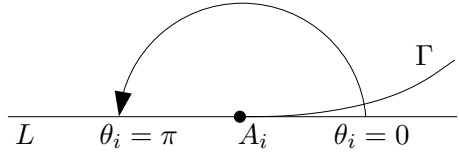}
\end{center}
\caption{Polar coordinates with origin $A_i$, and such that $\theta_i=0$ corresponds to the semi-axis tangent to $\Gamma$.}\label{polarfig}
\end{figure}

These singularities are problematic for two reasons. The first difficulty is numerical: these singularities may produce inacurracies when computing the solution near the points $\{A_1,A_2\}$, unless the proper numerical setting is used. It also possibly produces non-smooth deformations of the shape, which might create in turn undesired angles in the shape during the optimization procedure. The second difficulty is theoretical: since $\Gamma$ is a free boundary with the constraint $\Om\subset\mathds{R}^N_+$, the points $\{A_1,A_2\}$ are also "free points", i.e. their optimal position is unknown in the same way as $\Gamma$ is unknown. This means that the sensitivity with respect to those points has to be studied, which is doable but tedious, although interesting. The main ingredient in the computation of the shape sensitivity with respect to these points is the evaluation of the stress intensity factors $c(A_i)$. \\

\subsection{Penalization of the partial differential equation}

In order to deal with the aforementionned issue, we introduce a fourth approach, based on the penalization of the jump in the boundary conditions \eqref{6.3}-\eqref{6.4} for $u_2$. Let $\e\geq 0$ be a small real parameter, and let $\psi_\e\in\mathcal{C}(\mathds{R}^+,\mathds{R}^+)$ be a decreasing penalization function such that $\psi_\e\geq 0$, $\psi_\e$ has compact support $[0,\beta_\e]$, and with the properties
\begin{align}
 \label{6.5aa}\beta_\e\to 0 &\mbox{ as }\e\to 0,\\
 \label{6.5a}\psi_\e(0)\to \infty &\mbox{ as }\e\to 0,\\
 \label{6.5b}\psi_\e(x_1)\to 0 &\mbox{ as }\e\to 0,\quad\forall x_1>0 .
\end{align}
A simple example of such function is given by
\begin{equation}
\label{6.15} \psi_\e(x_1)=\e^{-1}(\max(1-\e^{-q}x_1,0))^2\mathds{1}_{\mathds{R}^+},
\end{equation}
with $q>0$. Note that $\psi_\e$ is decreasing, has compact support and verifies assumptions \eqref{6.5aa}-\eqref{6.5b}, with $\beta_\e=\e^q$. We will see in Proposition \ref{prop1} that the choice of $\psi_\e$ is conditioned by the shape of the domain. Then we consider the problem with Robin boundary conditions
\begin{align}
 \label{6.5}-\Delta \ub & =  0 \quad  \mbox{in} \  \Om , \\
 \label{6.6}\ub & =  1 \quad  \mbox{on}\  K ,\\
 \label{6.7}\dn \ub+\psi_\e(x_1)\ub & =-1 \quad  \mbox{on} \ \partial\Om\setminus K .
\end{align}
The function $\ub$ is a penalization of $u_2$ in the sense that $\ub\to u_2$ as $\e\to 0$ in $H^1(\Om)$ if $\psi_\e$ is properly chosen. The following Proposition ensures the $H^1$-convergence of $\ub$ to the desired function. It may be noticed that an explicit choice of function $\psi_\e$ providing the convergence is given in the statement of this Proposition.
\begin{proposition}\label{prop1}
Let $\Om$ be an open bounded domain. Then for $\psi_\e$ given by \eqref{6.15}, there exists a unique solution to \eqref{6.5}-\eqref{6.7} which satisfies
\begin{equation}
 \ub\to u_2\ \mbox{ in }H^1(\Om)\ \mbox{ as }\e\to 0.
\end{equation}
\end{proposition}
\begin{proof}
In the sequel, $c$ will denote a generic positive constant which may change its value throughout the proof and does not depend on the parameter $\e$.
\par We shall prove that the difference 
$$
v_\e=u_2-\ub .
$$
converges to zero in $H^1(\Om)$. The remainder $v_\e$ satisfies, according to \eqref{6.1}-\eqref{6.4} and \eqref{6.5}-\eqref{6.7}
\begin{align}
 \label{6.8}-\Delta v_\e & =  0 \quad  \mbox{in} \  \Om , \\
 \label{6.9}v_\e & =  0 \quad  \mbox{on}\  K ,\\
 \label{6.10}\dn v_\e+\psi_\e(0)v_\e & = 1+\dn u_2  \quad  \mbox{on}\  L ,\\
 \label{6.11}\dn v_\e+\psi_\e(x_1)v_\e & = \psi_\e(x_1) u_2 \quad  \mbox{on} \ \Gamma .
\end{align}
Multiplying by $v_\e$ on both sides of \eqref{6.8}, integrating on $\Om$ and using Green's formula, we end up with 
\begin{equation}
\label{6.12}\int_\Om |\nabla v_\e|^2 +\int_{\partial\Om} (v_\e)^2\psi_\e  =\int_\Gamma \psi_\e u_2 v_\e +\int_L (1+\dn u_2)v_\e . 
\end{equation}
Since $v_\e=0$ on $K$ we may apply Poincar\'e's Theorem and \eqref{6.12} implies
\begin{equation}
\label{6.13}\nu\Vert v_\e\Vert^2_{H^1(\Om)} \leq  c\left(\Vert \psi_\e u_2 \Vert_{L^2(\Gamma)} \Vert v_\e\Vert_{L^2(\Gamma)} +  \Vert 1+\dn u_2 \Vert_{L^2(L)} \Vert v_\e\Vert_{L^2(L)}\right),
\end{equation}
According to the trace Theorem and Sobolev's imbedding Theorem, we have
\begin{align*}
\Vert v_\e\Vert_{L^2(\Gamma)} &\leq c\Vert v_\e\Vert_{H^{1/2}(\Gamma)}\leq c\Vert v_\e\Vert_{H^{1}(\Om)}, \\
\Vert v_\e\Vert_{L^2(L)}& \leq c\Vert v_\e\Vert_{H^{1/2}(L)}\leq c\Vert v_\e\Vert_{H^{1}(\Om)}.
\end{align*}
Hence, according to \eqref{6.13}, we get
\begin{equation}
\label{6.14}\Vert v_\e\Vert_{H^1(\Om)} \leq c \Vert \psi_\e u_2 \Vert_{L^2(\Gamma)} + c\Vert 1+\dn u_2 \Vert_{L^2(L)}.
\end{equation}
Now we prove that $\Vert \psi_\e u_2 \Vert_{L^2(\Gamma)}\to 0$ as $\e\to 0$.  We may assume that the system of cartesian coordinates $(O,x_1,x_2)$ is such that the origin $O$ is one of the points $A_1$ or $A_2$ and that $\Gamma$ is locally above the $x_1$-axis; see Figure \ref{concave}.
\begin{figure}
\begin{center}
\psfrag{a}{$x_1$}
\psfrag{b}{$x_2$}
\psfrag{c}{$\Gamma$}
\psfrag{d}{$A_i$}
\includegraphics[width=6cm]{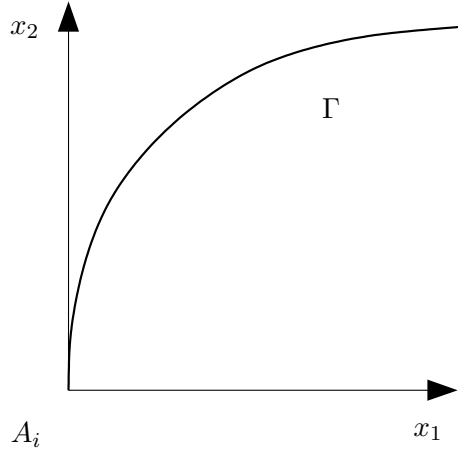}
\end{center}
\caption{$\Gamma$ is locally the graph of a convex function, with a tangent to the $x_2$-axis.}\label{concave}
\end{figure}
Since $\Om$ is convex, there exist $\delta>0$ and two constants $\alpha>0$ and $\beta$ such that for all $x_1\in (0,\delta)$, $\Gamma$ is the graph of a convex function $f$ of $x_1$. For our choice of $\psi_\e$, since $\textrm{supp}~\psi_\e=[0,\beta_\e]$, we have the estimate
\begin{align*}
\Vert \psi_\e u_2 \Vert_{L^2(\Gamma)}^2 & = \int_\Gamma(\psi_\e u_2)^2\\
& \leq \psi_\e(0)^2\int_0^{\be}(u_2)^2\sqrt{1+f'(x_1)^2}\, dx_1.
\end{align*}
According to \cite{MR775683,MR0226187} and our previous remarks in section \ref{sop}, we have $u_2=\sqrt{r}\cos(\theta/2)+u_\infty$, with $u_\infty\in H^2(\Om)$, and $(r,\theta)$ are the polar coordinates defined previously with origin $0$. Thus there exists a constant $c$ such that
$$|u_2|\leq c\sqrt{r}\cos(\theta/2)$$
in a neighborhood of $0$ with $\theta\in (0,\pi/2)$. Indeed, $u_\infty$ is $H^2$ therefore $C^1$ in a neighborhood of 0 and then has an expansion of the form: $u_\infty=c_s r+o(r)$, as $r\to 0$. Note that $r=\sqrt{x_1^2+x_2^2}$ and thus $r=\sqrt{x_1^2+f(x_1)^2}$ on $\Gamma$. Then
\begin{align*}
\Vert \psi_\e u_2 \Vert_{L^2(\Gamma)}& \leq c\psi_\e(0)\left(\int_0^{\be}(\sqrt{r}\cos(\theta/2))^2\sqrt{1+f'(x_1)^2}\, dx_1\right)^{1/2}\\
& \leq c\psi_\e(0)\left(\int_0^{\be}\sqrt{(x_1^2+f(x_1)^2)(1+f'(x_1)^2)}\, dx_1\right)^{1/2}.
\end{align*}
The function $f$ is convex and $f(0)=0$, thus $f'>0$ for $\e$ small enough. Since the boundary $\Gamma$ is tangent to the $(Ox_2)$ axis, we have 
\begin{align*}
f'(x_1)\to \infty &\ \mbox{ as }x_1\to 0^+,\\
x_1=o(f(x_1)) &\ \mbox{ as }x_1\to 0^+ .
\end{align*}
Thus, for $\e>0$ small enough
\begin{align*}
\Vert \psi_\e u_2 \Vert_{L^2(\Gamma)} & \leq c\psi_\e(0)\left(\int_0^{\be}f(x_1)f'(x_1)\, dx_1\right)^{1/2}\\
& \leq c\psi_\e(0)\left(f(\be)^2\right)^{1/2}= c\psi_\e(0) f(\be) .
\end{align*}
Since $f(x_1)\to 0$ as $x_1\to 0$, we may choose $\psi_\e(0)$ and $\be$ in order to obtain $\psi_\e(0) f(\be)\to 0$ as $\e\to 0$ and
\begin{equation}
\label{6.16} \Vert \psi_\e u_2\Vert_{L^2(\Gamma)}\to 0\quad \mbox{ as }\e\to 0.
\end{equation}
Then, in view of \eqref{6.14}, we may deduce that $\Vert v_\e\Vert_{H^1(\Om)}$ is bounded for the appropriate choice of $\psi_\e$. Consequently, $\Vert v_\e\Vert_{L^2(\Gamma)}$ and $\Vert v_\e\Vert_{L^2(L)}$ are also bounded. Using \eqref{6.12}, we may also write
\begin{align}
\notag \psi_\e(0) \Vert v_\e\Vert^2_{L^2(L)} &= \int_{L} (v_\e)^2\psi_\e \leq  \int_{\partial\Om} (v_\e)^2\psi_\e \\
\label{6.17} &  \leq \Vert \psi_\e u_2 \Vert_{L^2(\Gamma)} \Vert v_\e\Vert_{L^2(\Gamma)} + \Vert 1+\dn u_2 \Vert_{L^2(L)} \Vert v_\e\Vert_{L^2(L)}.
\end{align}
Since $\psi_\e(0) \to\infty$ as $\e\to 0$ and all terms in \eqref{6.17} are bounded, we necessarily have 
$$
\Vert v_\e\Vert_{L^2(L)}\to 0\mbox{ as }\e\to 0.
$$
Finally going back to \eqref{6.13} and using the previous results, we obtain 
$$ \Vert v_\e\Vert_{H^1(\Om)}\to 0\mbox{ as }\e\to 0,$$
and this proves  $\ub\to u_2$ as $\e\to 0$, in $H^1(\Om)$.
\end{proof}
The following theorem gives a mathematical justification of the numerical scheme implemented in section \ref{numScheme} to find the solution of the free Bernoulli problem $(\mathcal{F})$, based on the use of a penalized functional $J_\e$ defined by
\begin{equation}\label{6.18}
J_\e (\Om )=\int_{\Om}(u_{2,\e}-u_1)^2 ,
\end{equation}
where $u_1$ is the solution of (\ref{2.1})-(\ref{2.3}) and $u_{2,\e}$ is the solution of (\ref{6.5})-(\ref{6.7}).
 \begin{theorem}\label{thm1}
 One has
 $$
 \lim_{\e \to 0}\inf_{\Om\in\mathcal{O}}(J_\e(\Om)-J(\Om))=0. 
 $$
 \end{theorem}
 \begin{proof}
 The main ingredient of this proof is the result stated in Proposition \ref{prop1}. Indeed, this proposition yields in particular the convergence of $u_{2,\e}$ to $u_2$ in $L^2(\Om)$, when $\Om$ is a fixed element of $\mathcal{O}$. It follows immediately that
 $$
 J_\e (\Om)\to J(\Om),\textrm{ as }\e\to 0.
 $$
 Let us denote by $\Om^\star$ the solution of the free Bernoulli problem $(\mathcal{F})$. Then, we obviously have
 $$
 \inf_{\Om\in\mathcal{O}}J_\e(\Om)\leq J_\e(\Om^\star).
 $$
 Then, going to the limit as $\e \to 0$ yields
 $$
 0\leq \lim_{\e\to 0}\inf_{\Om\in\mathcal{O}}J_\e(\Om)\leq \lim_{\e\to 0}J_\e(\Om^\star)= J(\Om^\star)=0.
 $$
 \end{proof}
\begin{remark}
Theorem \ref{thm1} does not imply the existence of solutions for the problem $\inf \{J_\varepsilon (\Om), \Om\in \mathcal{O}\}$ and the following questions remain open: (i) existence of a minimizer $\Om_\varepsilon^\star$ for this problem, (ii) compactness of $(\Om_\varepsilon^\star)$ for an appropriate topology of domains. These problems appear difficult since to solve it, we probably need to establish a Sverak-like theorem for the Laplacian with Robin boundary conditions and some counter examples (see e.g. \cite{MR1458457}) suggest that this is in general not true.
\par Nevertheless, if (i) and (ii) are true, Theorem \ref{thm1} implies the convergence as $\varepsilon\to 0$, of $\Om_\varepsilon^\star$ to $\Om^\star$, the solution of  \eqref{1.1}-\eqref{1.4}.
\end{remark}

\section{Shape derivative for the penalized Bernoulli problem}\label{SectDer}

In order to stay in the class of domains $\OO$, the speed $V$ should satisfy
\begin{align}
\label{a1}V(x)=0 & \quad\forall x\in K, \\
\label{a2}V(x)\cdot n(x)<0 & \quad\forall x\in L.
\end{align}
Condition \eqref{a1} will be taken into account in the algorithm, and \eqref{a2} will be guaranteed by our optimization algorithm. We have the following result for the shape derivative  $dJ_\e(\Om;V)$ of $J_\e(\Om)$
\begin{theorem}\label{thm_shapeder}
The shape derivative  $dJ_\e(\Om;V)$ of $J_\e$ at $\Om$ in the direction $V$ is given by
\begin{eqnarray*}
dJ_\e(\Om;V)& =&\int_{\Gamma}\left(\nabla p_1\cdot\nabla u_1+\nabla p_2\cdot\nabla \ub+p_2\mathcal{H}+(u_1-\ub)^2\right) V\cdot n\, d\Gamma,\\
&& +\int_{L}\left(\nabla p_1\cdot\nabla u_1-\nabla p_2\cdot\nabla \ub\right) V\cdot n\, dL,
\end{eqnarray*}
where $\mathcal{H}$ is the mean curvature of $\Gamma$ and $p_1$, $p_2$ are given by \eqref{9.1}-\eqref{9.2} and \eqref{10.1}-\eqref{10.4}, respectively.
\end{theorem}
\begin{proof}
According to \cite{MR1855817,MR2512810,MR1215733}, the shape derivative of $J_\e$ is given by
\begin{equation}
\label{7.0} dJ_\e(\Om;V)=\int_{\Om}2(u_1-u_2)(u'_1-\ub')+\int_{\partial\Om}(u_1-\ub)^2V\cdot n,
\end{equation}
where $u'_1$ and $\ub'$ are the so-called {\it shape derivatives} of $u_1$ and $u_2$, respectively, and solve
\begin{align}
 \label{7.1}-\Delta u'_1 & =  0 \quad  \mbox{in} \  \Om , \\
 \label{7.2}u'_1 & =  0 \quad  \mbox{on}\  K ,\\
 \label{7.3}u'_1 & =  -\dn u_1V\cdot n \quad  \mbox{on}\  \partial\Om\setminus K ,
\end{align}
\begin{align}
 \label{8.1}-\Delta \ub' & =  0 \quad  \mbox{in} \  \Om , \\
 \label{8.2}\ub' & =  0 \quad  \mbox{on}\  K ,\\
 \label{8.3}\ub' & =  -\dn \ub V\cdot n  \quad  \mbox{on}\  L ,\\
\notag \dn \ub'+\psi_\e\ub' & =\mbox{div}_\Gamma (V\cdot n \nabla_\Gamma \ub)\\
 \label{8.4}&\quad -\mathcal{H}V\cdot n- \psi_\e\dn\ub V\cdot n\quad  \mbox{on} \  \Gamma,
\end{align}
where $\mathcal{H}$ denotes the mean curvature of $\Gamma$, and $\nabla_\Gamma$ is the tangential gradient on $\Gamma$ defined by
$$\nabla_\Gamma u=\nabla u -(\dn u )n. $$
Note that $u_1'$ and $\ub'$ both vanish on $K$, indeed, $K$ is fixed due to \eqref{a1} which follows from the  definition of our problem and of the class $\OO$. Further we will also need 
\begin{equation}
\label{8.5}\dn \ub'=\mbox{div}_\Gamma (V\cdot n \nabla_\Gamma \ub)-\mathcal{H}V\cdot n\quad  \mbox{on} \  \Gamma ,
\end{equation}
which is obtained in the same way as \eqref{8.4}. We introduce the adjoint states $p_1$ and $p_2$
\begin{align}
 \label{9.1}-\Delta p_1 & =  2(u_1-\ub)  \quad  \mbox{in} \  \Om , \\
 \label{9.2}p_1 & =  0 \quad  \mbox{on}\  \partial\Om ,
\end{align}
\begin{align}
 \label{10.1}-\Delta p_2 & =  2(u_1-\ub) \quad  \mbox{in} \  \Om , \\
 \label{10.2}p_2 & =  0 \quad  \mbox{on}\  L\cup K ,\\
 \label{10.4}\dn p_2 & =0 \quad  \mbox{on} \  \Gamma.
\end{align}
Note that $p_1$ and $p_2$ actually depend on $\e$ although this is not apparent in the notation for the sake of readability. Using the adjoint states, we are able to compute
\begin{eqnarray*}
\int_{\Om}2(u_1-\ub)u'_1 & = & \int_{\Om}-\Delta p_1u'_1\\
&= & \int_{\Om}-\Delta u'_1p_1-\int_{\partial\Om}\dn p_1u'_1-p_1\dn u_1\\
&=& -\int_{\partial\Om\setminus K}\dn p_1u'_1\\
&=&  \int_{\partial\Om\setminus K}\dn p_1\dn u_1 V\cdot n .
\end{eqnarray*}
Observing that $\nabla p_1=\dn p_1 n$ and $\nabla u_1=\dn u_1 n$ on $\partial\Om\setminus K$  due to \eqref{2.3} and \eqref{9.2} we obtain
\begin{equation}
\int_{\Om}2(u_1-\ub)u'_1dx =  \int_{\partial\Om\setminus K}\nabla p_1\cdot\nabla u_1 V\cdot n.
\end{equation}
For the other domain integral in \eqref{7.0} we get
\begin{eqnarray*}
\int_{\Om}2(u_1-\ub)\ub' & = & \int_{\Om}-\Delta p_2\ub'\\
&= & \int_{\Om}-\Delta \ub'p_2-\int_{\partial\Om}(\dn p_2\ub'-p_2\dn \ub') .\\
\end{eqnarray*}
At this point we make use of \eqref{8.1}-\eqref{8.5} and we get
\begin{eqnarray*}
\int_{\Om}2(u_1-\ub)\ub' &=& \int_{\Gamma} p_2(\mbox{div}_\Gamma (V\cdot n \nabla_\Gamma \ub)-\mathcal{H}V\cdot n)d\Gamma+ \int_{L}\dn p_2\dn \ub V\cdot n~dL .
\end{eqnarray*}
Applying classical tangential calculus to the above equation (see \cite[Proposition 2.57]{MR1215733} for instance) we have
\begin{eqnarray*}
\int_{\Om}2(u_1-\ub)\ub' &=&  -\int_{\Gamma}(\nabla_\Gamma p_2 \cdot\nabla_\Gamma \ub V\cdot n -p_2\mathcal{H}V\cdot n)d\Gamma+ \int_{L}\dn p_2\dn \ub V\cdot n ~dL\\
&=&  -\int_{\Gamma}(\nabla p_2 \cdot\nabla \ub V\cdot n -p_2\mathcal{H}V\cdot n)d\Gamma+ \int_{L}\nabla p_2\cdot\nabla \ub V\cdot n ~dL,\\
\end{eqnarray*}
and the proof is complete.
\end{proof}

\section{Numerical scheme}\label{numScheme}

\subsection{Parameterization versus level set method}

For the numerical realization of shape optimization problems, the main issue is the representation of the moving shape $\Om$. Several different techniques are available: for our purpose, the most appropriate methods would be {\it parameterization} and the {\it level set method}. In the parameterization method for two-dimensional problems, curves are typically represented as splines given by control points $\p_k=(\p_{1,k},\p_{2,k})$, $k=0,..,m$ with $m\in\mathds{N}^*$. The coordinates of these control points then become the shape design variables. In the level set method, the boundary of the domain in $\mathds{R}^N$ is implicitely given by the zero level set of a function in $\mathds{R}^{N+1}$. Parameterization methods are the easiest to implement if the topology of the domain $\Om$ does not change in the course of iterations, whereas the level set method is more technical to implement but thanks to the implicit representation, it allows to handle easily topological changes of the domain, such as the creation of holes or the merging of two connected components.\\

For instance, in \cite{MR2149216,MR2306261}, the level set method is used to solve Bernoulli free boundary problem where the number of connected components is not known beforehand. In our case, we are solving the free boundary problem $(\mathcal{F})$ in the class $\OO$ of convex domains, thus the domains only have one connected component and the topology is known. In this case it is better to opt for the parameterization method which is easier to implement and lighter in terms of computations.\\ 

The free boundary $\Gamma\subsetneq\partial\Om$ is represented with the help of a Bezier curve of degree $m\in\mathds{N}^*$. Let 
$$x(s)=(x_1(s),x_2(s)),\quad s\in [0,1]$$
be a parametric representation of the open curve  $\Gamma$ and let 
$$\p_k=(\p_{1,k},\p_{2,k}),\quad k=0,..,m$$
be a set of $m+1$ control points such that the parameterization of $\Gamma$ satisfies
\begin{equation}
\label{ns1} x(s)=(x_1(s),x_2(s))=\sum_{k=0}^m B_{k,m}(s) \p_k,
\end{equation}
where
\begin{equation}
\label{ns2} B_{k,m}(s)= \binom {m} {k} s^k(1-s)^{m-k},
\end{equation}
and $ \binom {m} {k}$ are the binomial coefficients. The geometric features such as the unit tangent $\tau(s)$, unit normal $n(s)$ and curvature $\mathcal{H}(s)$ are easily obtained from the representation \eqref{ns1}. Indeed we have
\begin{equation}
\label{ns3} \tau(s) = x'(s)/|x'(s)|,
\end{equation}
with 
\begin{equation}
\label{ns4} x'(s)=\sum_{k=0}^m B_{k,m}'(s) \p_k.
\end{equation}
The coefficients $B_{k,m}'(s)$ are derived from \eqref{ns2}
\begin{equation}
\label{ns5} B_{k,m}'(s) =  \binom {m} {k} \left[ks^{k-1}(1-s)^{m-k}\mathds{1}_{\{k\geq 1\}}+ (k-m)s^k(1-s)^{m-k-1}\mathds{1}_{\{k\leq m-1\}} \right] .
\end{equation}
Since $n(s)\cdot\tau(s)=0$, we deduce the expression for the unit normal $n(s)$
\begin{equation}
\label{ns6} n(s)=\frac{\sum_{k=0}^m B_{k,m}'(s) \p_k^\bot}{ \left|\sum_{k=0}^m B_{k,m}'(s) \p_k^\bot\right|},
\end{equation}
with $\p_k^\bot:=(\p_{2,k},-\p_{1,k})$. The curvature $\mathcal{H}(s)$ is obtained with the help of formula
\begin{equation}
\label{ns7} \tau'(s)=\mathcal{H}(s) n(s).
\end{equation}
Thus we take
\begin{equation}
\label{ns8} \mathcal{H}(s)=\tau'(s)\cdot n(s).
\end{equation}
\begin{remark}\label{ns8b}
According to \eqref{ns3}, \eqref{ns4} and \eqref{ns5}, we obtain
\begin{equation}
\tau(0) =\frac{\p_1-\p_0}{|\p_1-\p_0|},\quad\tau(1) =\frac{\p_m-\p_{m-1}}{|\p_m-\p_{m-1}|}.
\end{equation}
Thus, in order to create a curve which is tangent to the axis $\{x_1=0\}$, we need to take $\p_0,\p_1$ and $\p_{m-1},\p_m$ on $\{x_1=0\}$.
\end{remark}
\subsection{Algorithm}

For the numerical algorithm we use a gradient projection method in order to deal with the geometric constraint $\Om\subset\mathds{R}^N_+$; see the textbooks \cite{MR1678201,MR2244940} for details on the method. A solution for dealing with the shape optimization problems with a convexity constraint is to parameterize the boundary using a support function $w$. If one uses a polar coordinates representation $(r,\theta)$ for the domains, namely
$$
\Omega_w:=\left\{(r,\theta)\in [0,\infty)\times \mathds{R}; r< \frac{1}{w(\theta)} \right\},
$$
where $w$ is a positive and $2\pi$-periodic function, then $\Omega_w$ is convex if and only if $w'' + w\geq 0$; see \cite{Lamboley} for details. However, in our case, the convexity constraint for $\Om$ is not implemented (i.e. we relax this constraint) for the sake of simplicity, but the convexity property is observed at every iteration and in particular for the optimal domain if the initial domain is convex. Moreover, Theorem 6.6.2 of \cite{MR2512810} may be easily generalized in our case and guarantees the convexity of the solution of the free boundary problem $(\mathcal{F})$ even if the convexity hypothesis were not contained in the set $\mathcal{O}$. 
\par We will denote by a superscript $(l)$ an object at iteration $l$. The algorithm is as follows: we are looking for an update of the design variable $\p_k$ of the type
\begin{equation}
\label{ns9} \p_k^{(l+1)}=P(\p_k^{(l)}+\alpha d\p_k^{(l)}),
\end{equation}
where $P$ stands for the projection on the set of constraints and $\alpha$ is the steplength which has to be determined by an appropriate linesearch. In our case, the constraint is $\Om\subset\mathds{R}^N_+$, which implies the constraint
\begin{equation}
\label{ns9b} x_1(s)\geq 0,\quad\forall s\in [0,1].
\end{equation}
In view of \eqref{ns1}, it is difficult to directly interpret the constraint \eqref{ns9b} for individual control points $\p_k$. We choose therefore to impose the stronger constraint
\begin{equation}
\label{ns9c} \p_{1,k}\geq 0,\quad\forall k\in\{1,..,m\}.
\end{equation}
for the control points. Constraint \eqref{ns9c} is stronger than \eqref{ns9b}, indeed, on one hand there might exist a $\p_{k}$ such that $\p_{1,k}<0$ while \eqref{ns9b} is still  satisfied, but on the other hand, condition \eqref{ns9c} implies \eqref{ns9b}. However, in our case, the tips $x(0)$ and $x(1)$ of $\Gamma$ are moving and the constraint should not be active for the points of $\Gamma$ on the optimal domain. With \eqref{ns9c} we only guarantee that the domain stays feasible, i.e. $\Om\in\mathds{R}^N_+$ for all iterates. In view of Remark \ref{ns8b}, we also impose  
$$\p_{2,0}=\p_{2,1}=\p_{2,m-1}=\p_{2,m}=0$$
in order to preserve the tangent to the axis $\{x_1=0\}$ at the tips of $\Gamma$. Therefore, for $k=0,..,m$,  $\p_k^{(l)}$ is updated using,
\begin{align}
\label{ns10} \p_{1,k}^{(l+1)} &=\max\left(\p_{1,k}^{(l)}+\alpha d\p_{1,k}^{(l)},0\right),\\
\label{ns11} \p_{2,k}^{(l+1)} &=\p_{2,k}^{(l)}+\alpha d\p_{2,k}^{(l)},\\
\label{ns11b} d\p_{2,0}^{(l)}=d\p_{2,1}^{(l)} &=0,\\
\label{ns11c} d\p_{2,m-1}^{(l)}=d\p_{2,m}^{(l)} &=0.
\end{align}
The link between the perturbation field $V$ and the step $d\p_k$ is directly established using \eqref{ns1}, and we obtain 
\begin{equation}
\label{ns12} V(x(s))=\sum_{k=0}^m B_{k,m}(s) d\p_k.
\end{equation}
Thus, with a shape derivative given by 
\begin{equation}
\label{ns13} dJ_\e(\Om;V)=\int_{\partial\Om} \nabla J_\e(x) V(x)\cdot n(x)\, d\Gamma(x)
\end{equation}
as in Theorem \ref{thm_shapeder}, we obtain using \eqref{ns12} and \eqref{ns13}
\begin{align*}
dJ_\e(\Om;V) &= \int_{0}^1  \nabla J_\e(x(s)) V(x(s))\cdot n(s)|x'(s)|\, ds\\
&= \int_{0}^1  \nabla J_\e(x(s)) \left[\sum_{k=0}^m B_{k,m}(s) d\p_k\right]   \cdot n(s)|x'(s)|\, ds\\
&= \sum_{k=0}^m d\p_k \cdot\int_{0}^1  \nabla J_\e(x(s))  B_{k,m}(s)  n(s)|x'(s)|\, ds.
\end{align*}
Thus, a descent direction for the algorithm is given by
\begin{equation}
\label{ns14}  d\p_k=-\int_{0}^1  \nabla J_\e(x(s))  B_{k,m}(s)  n(s)|x'(s)|\, ds,
\end{equation}
and the update is then performed according to \eqref{ns10}-\eqref{ns11c}. The step $\alpha$ is determined by a line search in the spirit of the gradient projection algorithm \cite{MR1678201}: a step is validated if we observe a sufficient decrease of the shape functional $J_\e$ measured by 
$$J_\e(\Om^{(l+1)})-J_\e(\Om^{(l)})\leq -\frac{\alpha}{\lambda} \sum_{k=1}^m\vert \p_{k}^{(l+1)}- \p_{k}^{(l)} \vert^2,$$
where $\vert\cdot \vert$ denotes the Euclidian distance. The line search consists in finding the smallest integer $a$ (the smallest possible being $a=0$) such that 
$$\alpha=\mu\eta^a,$$
where $\mu$ and $\eta<1$ are user-defined parameters. To stop the algorithm, we use the following stopping criterion: we stop when
$$\vert \p_{k}^{(l+1)}- \p_{k}^{(l)} \vert \leq \tau_r\vert \p_{k}^{(1)}- \p_{k}^{(0)} \vert,$$
where $\tau_r$ is a user-defined parameter.

\section{Numerical results}\label{numres}

For the numerical resolution we take $m=40$ control points $\p_k$. We discretize the interval $[0,1]$ for the parameterization $x(s)$ using $400$ points. The domain $K$ is chosen as
$$K=\{0\}\times [0.5-\kappa_1,0.5+\kappa_1],$$
with 
$\kappa_1\approx 0.129$. The initial domain $L$ is chosen as
$$L=\{0\}\times [0.5-\kappa_2,0.5-\kappa_1]\cup [0.5+\kappa_1,0.5+\kappa_2],$$
with 
$\kappa_2\approx 0.233$. We use the Matlab PDE toolbox to produce a grid in $\Om$ and solve $u_1, \ub,p_1,p_2$ using finite elements. The geometric quantities such as tangent, normal and curvature are computed using \eqref{ns3}-\eqref{ns4}, \eqref{ns6} and \eqref{ns8}, respectively. We initialize the points $\p_k$ by placing them evenly on a half-circle of center $\{0\}\times \{0.5\}$ and radius $0.3$, except for the two first $\p_0,\p_1$ and two last points $\p_{m-1},\p_m$ which have to lay on the axis $\{x_1=0\}$ as mentionned earlier. We choose $\mu=10$, $\eta=0.5$ for the line search and $\tau_r=5\times 10^{-4}$ for the stopping criterion. For the penalization we use \eqref{6.15} and choose $\e=10^{-1}$ and $q=4$.\\

The algorithm terminated after $220$ iterations. The results are given in Figures \ref{fig_state} to \ref{fig_omega}.  In Figure  \ref{fig_state}, the two states $u_1$ and $\ub$ as well as the two adjoint states $p_1$ and $p_2$ are plotted. The difference between $u_1$ and $\ub$ in the final domain $\Om_{final}$ is plotted in Figure \ref{fig_residual}, along with the residual $J_\e(\Om)$ given by \eqref{6.18}. In Figure \ref{fig_omega}, the initial and final boundaries are plotted in blue and red, respectively, while the set of control points of the curve $\Gamma$ is plotted in green.  We observe  that the optimal domain is symmetric as expected from section \ref{symm}. The optimal set $L_{final}$ is given by
$$L_{final}=\{0\}\times [0.5-\kappa_{fi},0.5-\kappa_1]\cup [0.5+\kappa_1,0.5+\kappa_{fi}].$$
with
$\kappa_{fi}\approx 0.2342$. The value of $J_\e$ on the initial domain is
$$J_\e(\Om_{initial})\approx 2.6\times 10^{-3}, $$
and the value of $J_\e$ on the final domain is
$$J_\e(\Om_{final})\approx 3.3\times 10^{-8},$$
as may be seen in Figure \ref{fig_residual}. Therefore, the shape functional $J_\e$ has been significantely decreased and is close to its global optimum. \ \\

\vspace{1cm}
\textbf{Acknowledgments.}
The authors would like to express a great deal of gratitude to Professor Michel Pierre for several light brighting discussions.
The authors further acknowledge financial support by the Austrian Ministry of Science and Education and the Austrian Science Fundation FWF under START-grant Y305 ``Interfaces and free boundaries''. The second author were partially supported by the ANR project GAOS ``Geometric analysis of optimal shapes''.

\begin{figure}
\begin{center}
\includegraphics[width=7.5cm]{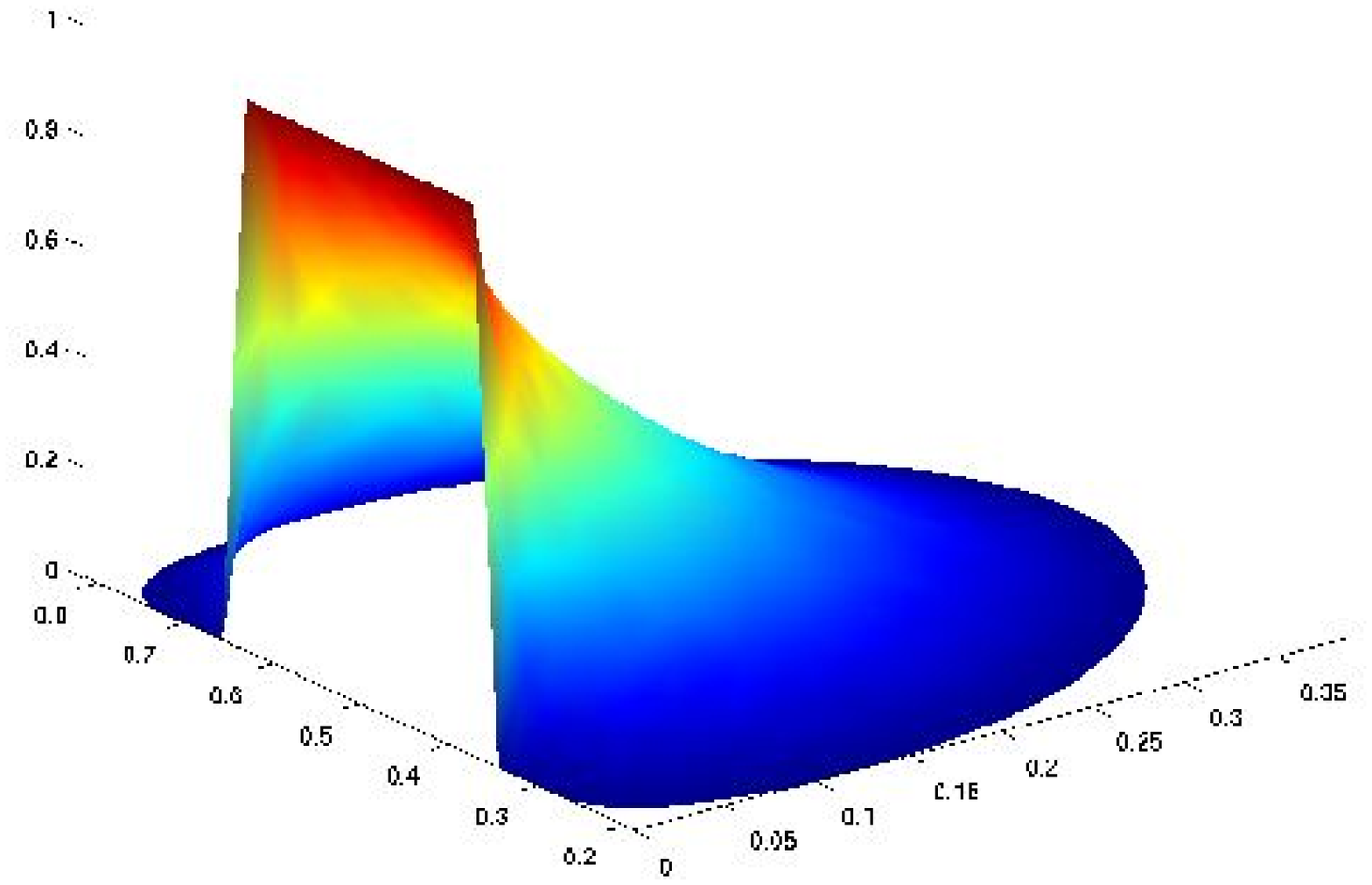}
\includegraphics[width=7.5cm]{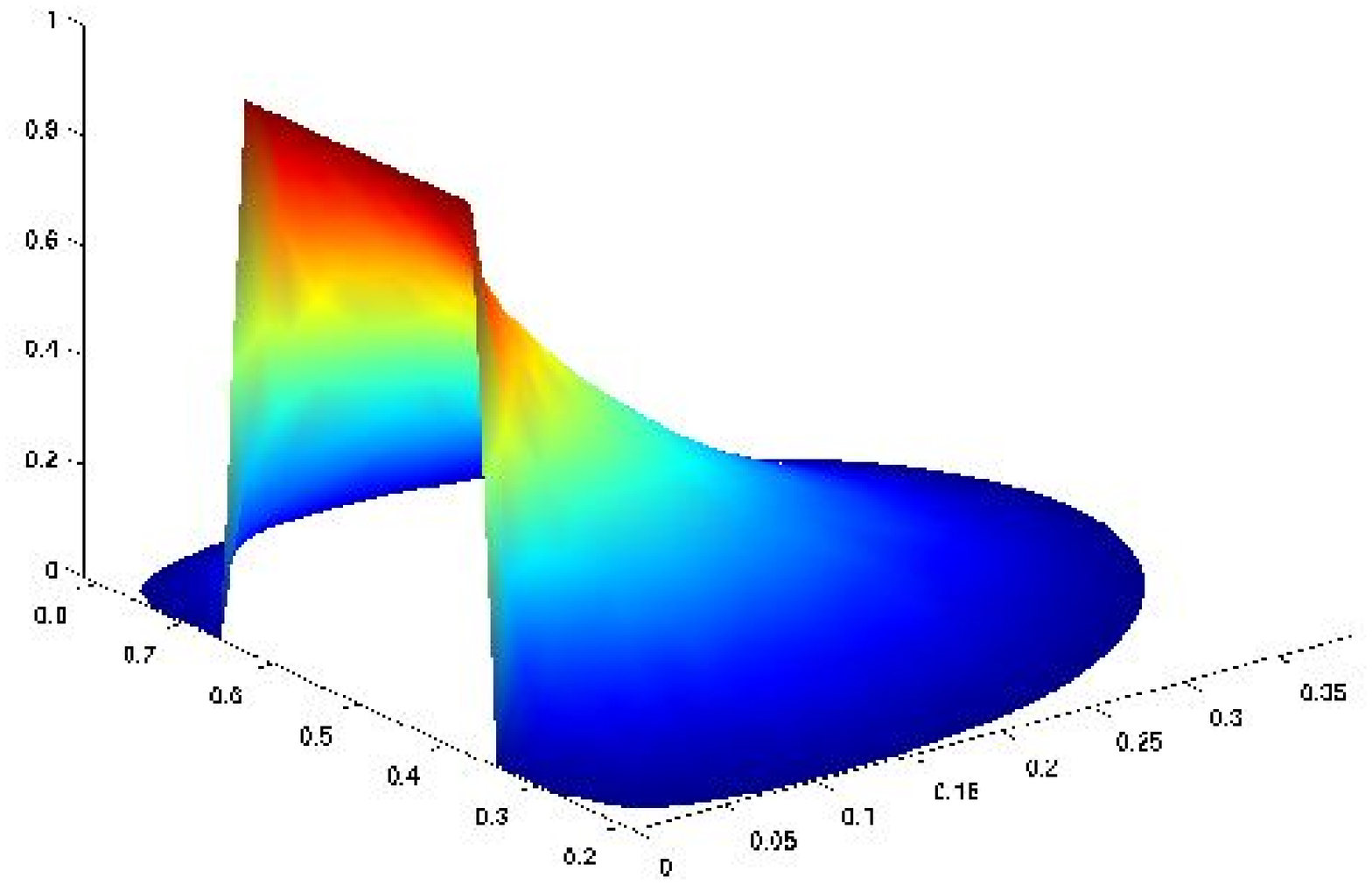}\\
\includegraphics[width=7.5cm]{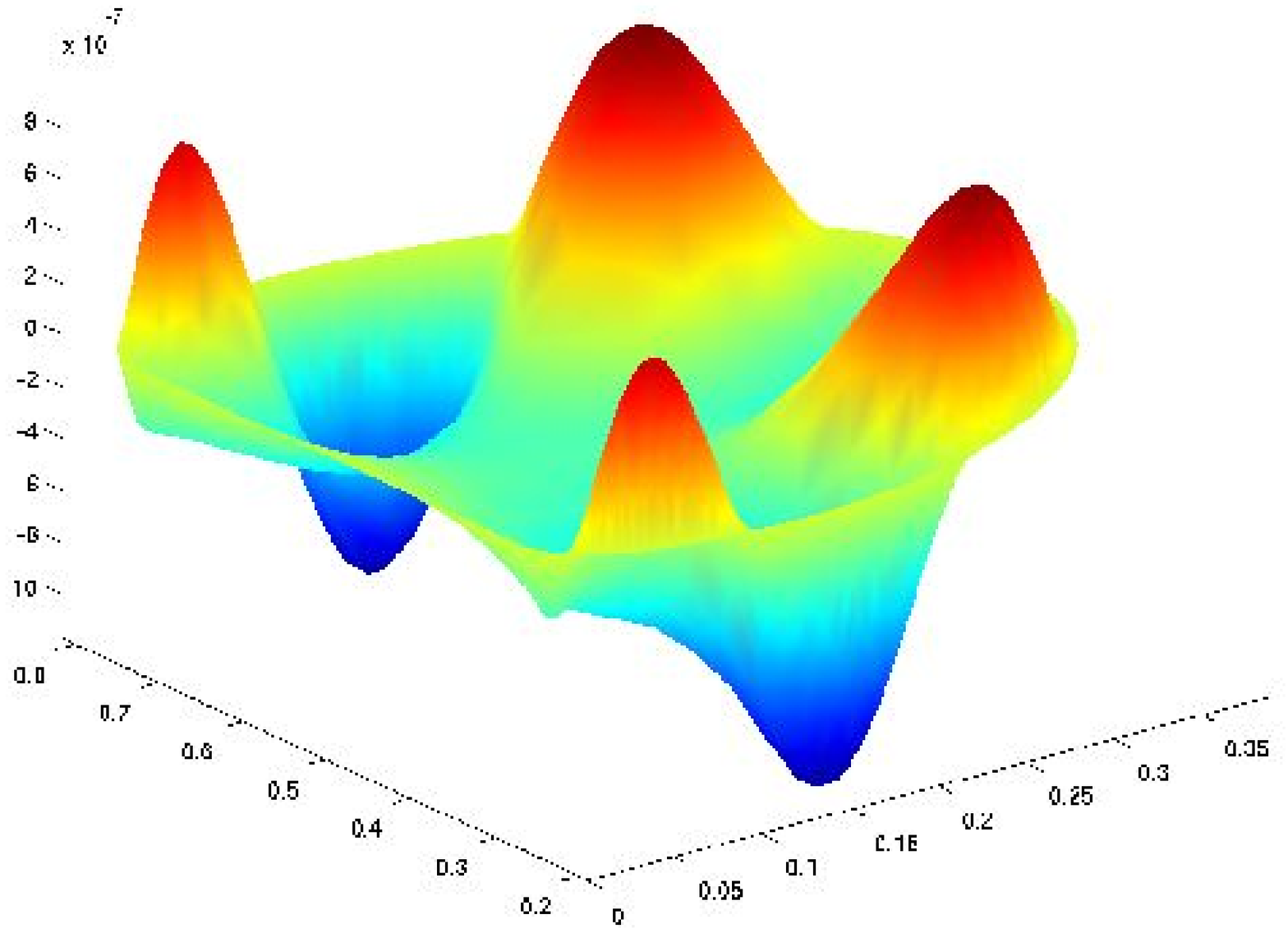}
\includegraphics[width=7.5cm]{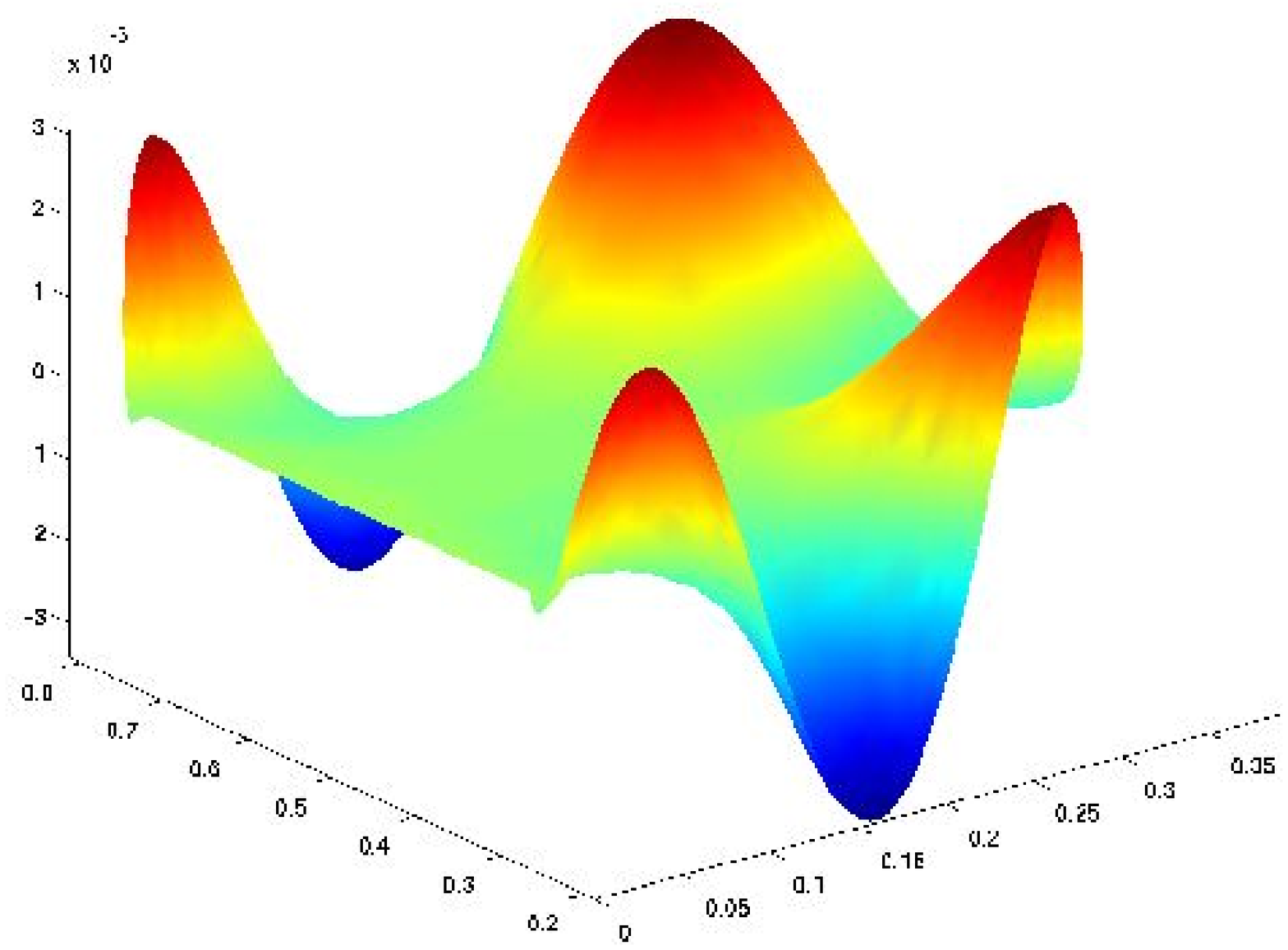}\\
\caption{Solutions $u_1$ (top left), $\ub$ (top right), $p_1$ (bottom left), $p_2$ (bottom right) in the optimal domain.}
\label{fig_state}
\end{center}
\end{figure}

\begin{figure}
\begin{center}
\includegraphics[width=7.5cm]{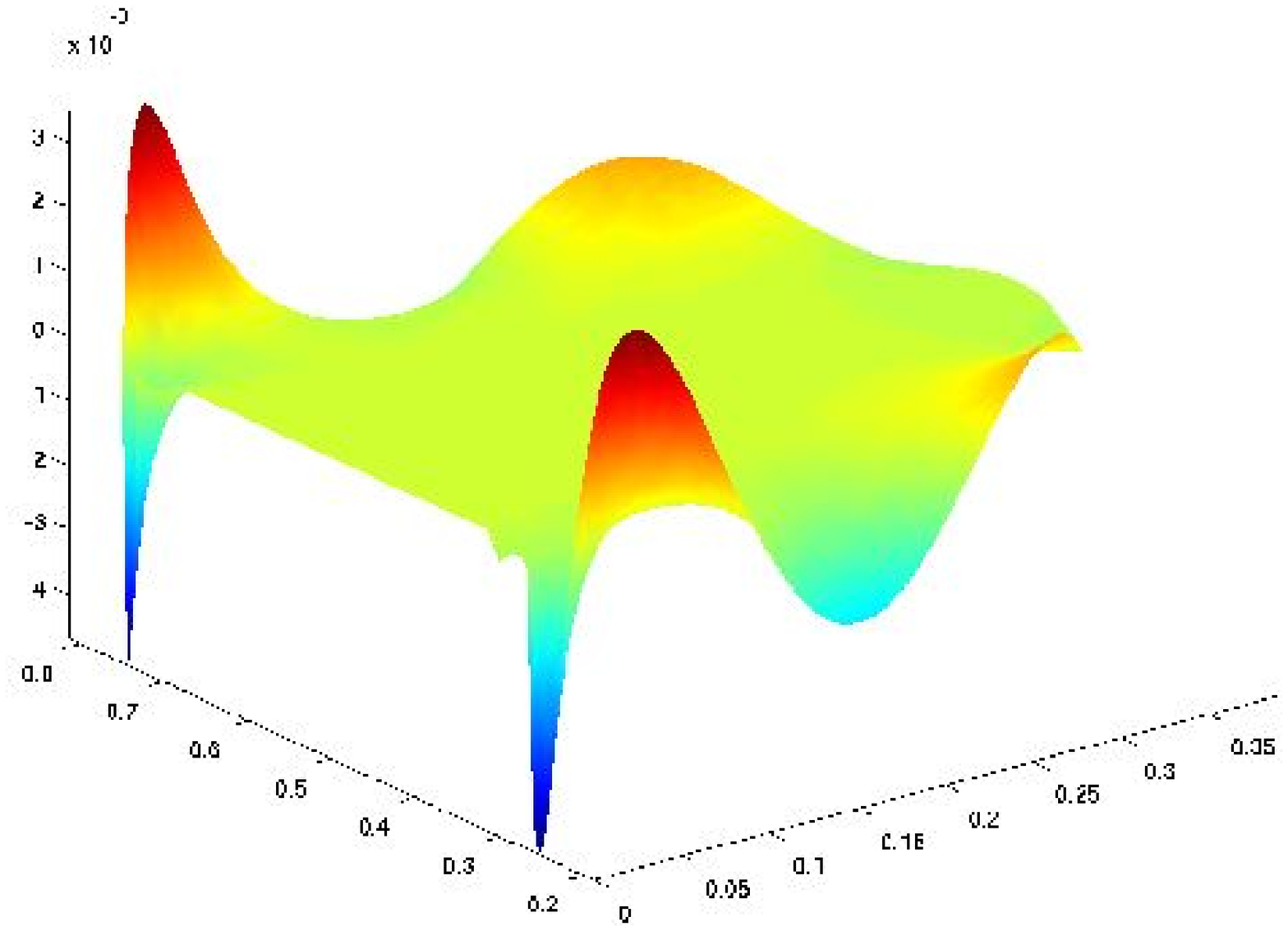}
\includegraphics[width=7.5cm]{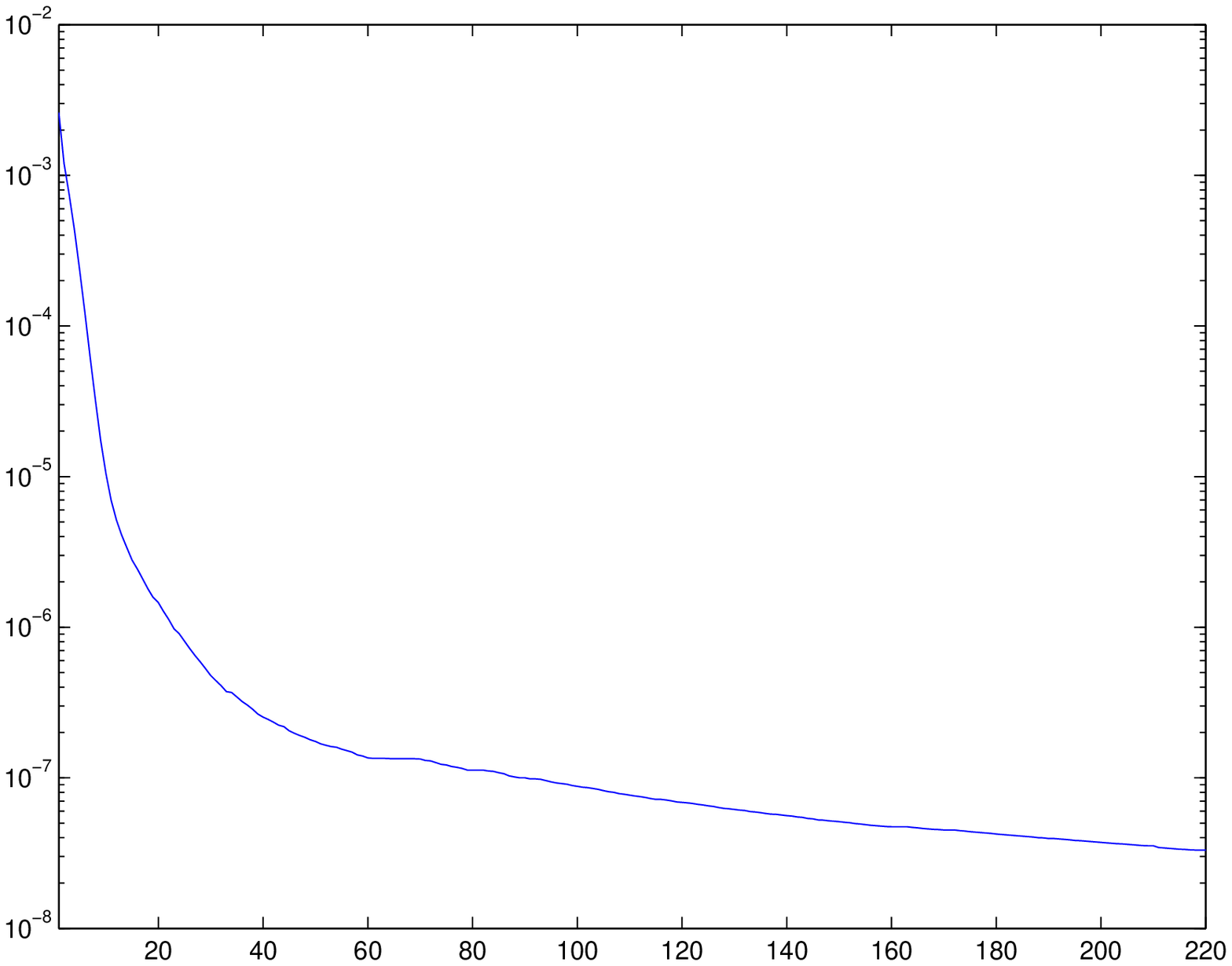}
\caption{Difference $u_1-u_2$ in the optimal domain (left), residual $J_\e$ (right).}
\label{fig_residual}
\end{center}
\end{figure}

\begin{figure}
\begin{center}
\includegraphics[width=7.5cm]{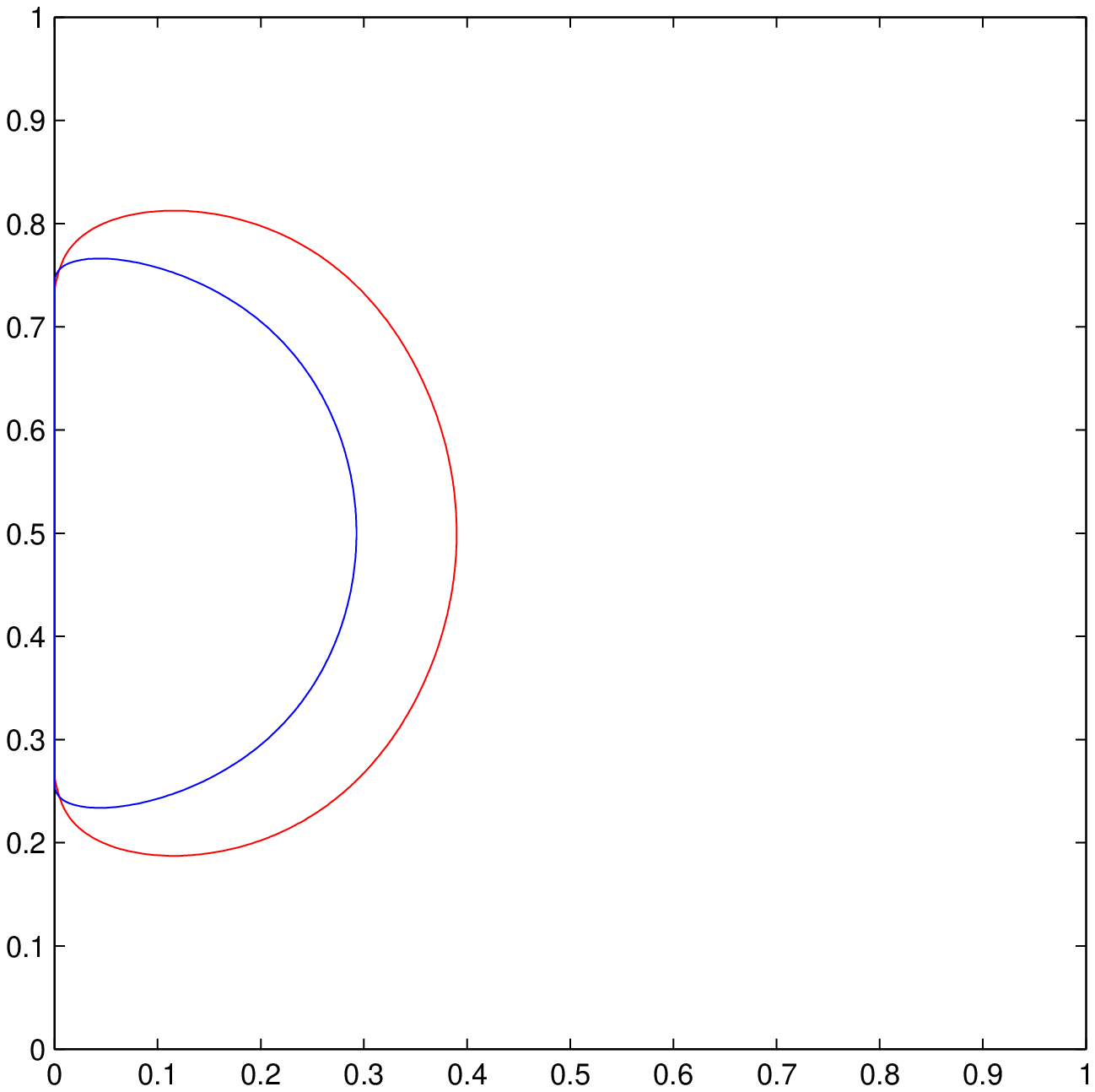}
\includegraphics[width=7.5cm]{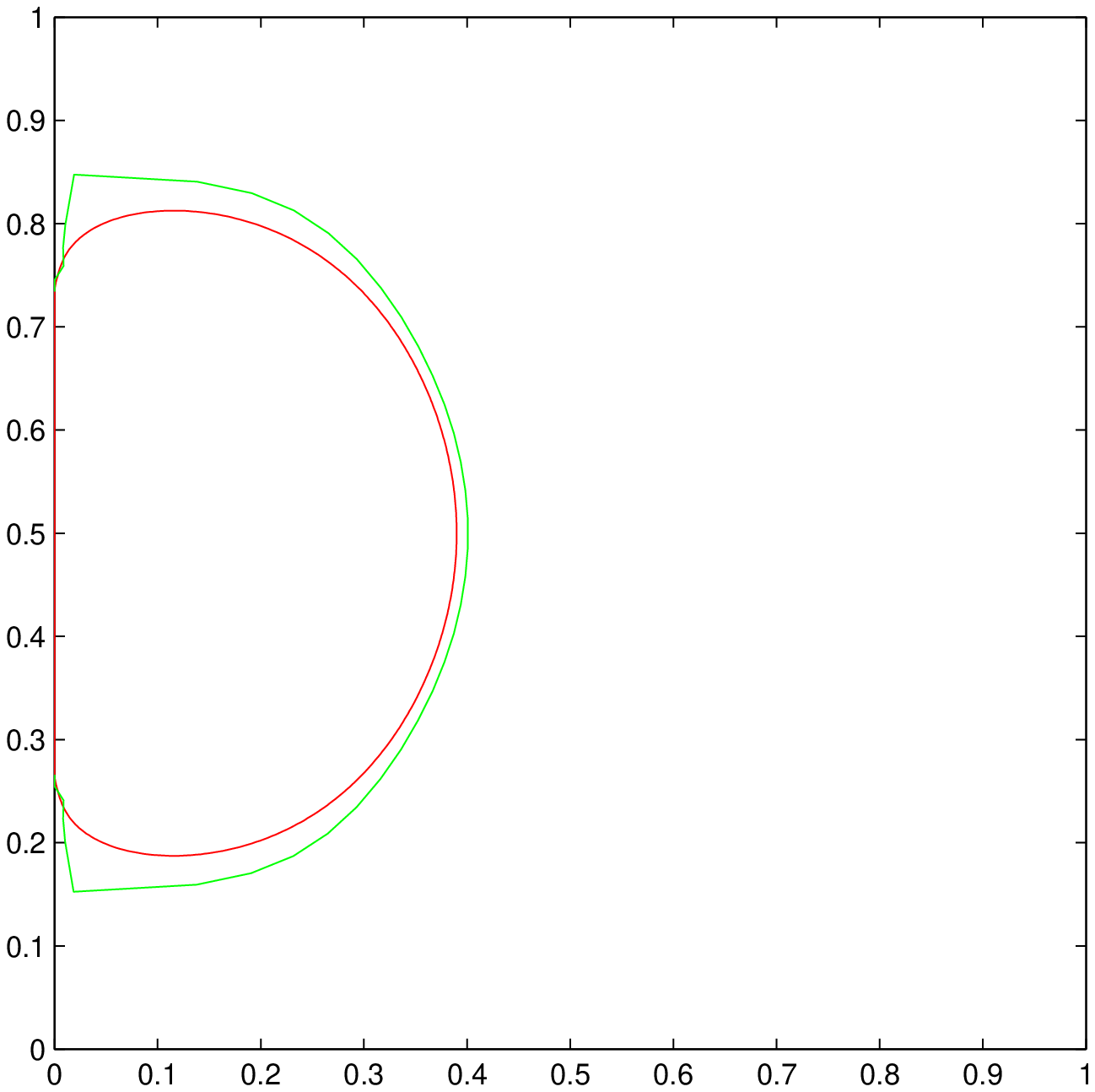}
\caption{Final boundary $\Gamma$ (red), initial boundary $\Gamma$ (blue), control points (green).}
\label{fig_omega}
\end{center}
\end{figure}

\def\cprime{$'$}

\end{document}